\documentclass[11pt]{amsart}

\usepackage{hyperref}
\usepackage{amsthm,amssymb}

\theoremstyle{theorem}
\newtheorem{theorem}{Theorem}[section]
\newtheorem{proposition}{Proposition}[section]
\newtheorem{corollary}{Corollary}[section]
\newtheorem{lemma}{Lemma}[section]
\newtheorem{claim}{Claim}[section]

\theoremstyle{definition}
\newtheorem{problem}{Problem}[section]
\newtheorem{definition}{Definition}[section]
\newtheorem{remark}{Remark}[section]

\begin{document}

\title[Rogers Semilattices in the Analytical Hierarchy: Finite Families]{Rogers Semilattices in the Analytical Hierarchy: The Case of Finite Families}

\author{Nikolay Bazhenov}
\address{Sobolev Institute of Mathematics, 4 Acad. Koptyug Avenue, Novosibirsk, 630090, Russia}
\address{Novosibirsk State University, 2 Pirogova Street, Novosibirsk, 630090, Russia}
\email{bazhenov@math.nsc.ru}

\author{Manat Mustafa}
\address{Department of Mathematics, School of Sciences and Humanities, Nazarbayev University, 53 Qabanbay Batyr Avenue, Nur-Sultan, 010000, Kazakhstan}
\email{manat.mustafa@nu.edu.kz}

\begin{abstract}
	A numbering of a countable family $S$ is a surjective map from the set of natural numbers $\omega$ onto $S$. The paper studies Rogers semilattices, i.e. upper semilattices induced by the reducibility between numberings, for families $S\subset P(\omega)$. Working in set theory \textbf{ZF}$+$\textbf{DC}$+$\textbf{PD}, we obtain the following results on families from various levels of the analytical hierarchy.
	
	For a non-zero number $n$, by $E^1_n$ we denote $\Pi^1_n$ if $n$ is odd, and $\Sigma^1_n$ if $n$ is even. We show that for a finite family $S$ of $E^1_n$ sets, its Rogers $E^1_n$-se\-mi\-lat\-ti\-ce has the greatest element if and only if $S$ contains the least element under set-theoretic inclusion. Furthermore, if $S$ does not have the $\subseteq$-least element, then the corresponding Rogers $E^1_n$-semilattice is upwards dense.
\end{abstract}


\maketitle

\section{Introduction}

Let $\mathcal{S}$ be a countable set. A \emph{numbering} of $\mathcal{S}$ is a surjective map $\nu$ from the set of natural numbers $\omega$ onto $\mathcal{S}$. The origins of the theory of numberings can be traced back to the works of G{\"o}del~\cite{Goedel} and Kleene~\cite{Kleene}. The proof of G{\"o}del's incompleteness theorems uses an effective numbering of first-order formulae. Kleene (see Theorem~XXII in \S\,65 of Ref.~\cite{Kleene}) gave a construction of a universal partial computable function~--- this result provides a universal computable numbering for the family of all unary partial computable functions. At the end of 1950s, the foundations of the modern theory of numberings were developed by Kolmogorov and Uspenskii~\cite{KU,Usp-55} and, independently, by Rogers~\cite{Rogers}.

The algorithmic complexity of different numberings is typically compared via the notion of \emph{reducibility} between numberings: A numbering $\nu$ is \emph{reducible} to a numbering $\mu$, denoted by $\nu\leq \mu$, if there is total computable function $f(x)$ such that $\nu(n) = \mu(f(n))$ for all $n\in\omega$. More informally, there is an effective procedure which, given a $\nu$-index of an object from $\mathcal{S}$, computes a $\mu$-index for the same object.

Since the end of 1960s, the research in the theory of numberings has been mainly focused in the area of \emph{Rogers semilattices}. We give a very brief overview of the classical setting in this area. From now on, we consider only families $\mathcal{S}$ containing subsets of $\omega$, i.e., we always assume that $\mathcal{S} \subset P(\omega)$ and $\mathcal{S}$ is countable.

Let $\mathcal{S}$ be a family of computably enumerable (c.e.) sets. A numbering $\nu$ of the family $\mathcal{S}$ is \emph{computable} if the set
\begin{equation} \label{equ:g-nu}
	G_{\nu} = \{ \langle n,x\rangle \,\colon x\in \nu(n)\}
\end{equation}
is c.e. A family $\mathcal{S}$ is \emph{computable} if it has a computable numbering. In other words, the computability of $\mathcal{S}$ means that one can \emph{uniformly enumerate} all sets from $\mathcal{S}$. Note that in general, this enumeration \emph{allows} repetitions.

As simple examples of computable families, one can immediately recall the family of all finite sets and the family of all c.e. sets. A more delicate example can be constructed as follows. If a number $e\not\in \emptyset'$, then our family $\mathcal{T}$ contains one-element sets $\{ 2e\}$ and $\{ 2e+1\}$. If $e\in\emptyset'$, then the set $\{ 2e,2e+1\}$ belongs to $\mathcal{T}$. It is easy to see that the constructed family $\mathcal{T}$ admits a uniform enumeration and thus, $\mathcal{T}$ is computable. A more interesting feature of $\mathcal{T}$ is the following: \emph{Any} computable numbering $\nu$ of $\mathcal{T}$ must have repetitions, i.e. there are indices $m\neq n$ with $\nu(m)=\nu(n)$.

In a standard recursion-theoretic way, the notion of reducibility between numberings gives rise to the corresponding upper semilattice:
For a computable family $\mathcal{S}$, the \emph{Rogers semilattice} of $\mathcal{S}$ contains the degrees of all computable numberings of $\mathcal{S}$. As per usual, here two numberings have the same degree if they are reducible to each other. Roughly speaking, the supremum of two numberings is provided by their join, see Section~\ref{subsect:numb} for formal details.

To give a flavor of studies of computable families, we mention here two celebrated classical results on Rogers semilattices: Let $\mathcal{S}$ be a computable family, and let $R$ be its Rogers semilattice. Khutoretskii~\cite{Khut-71} proved that if $R$ contains more than one element, then $R$ is infinite. Selivanov~\cite{Sel-76} showed that an infinite $R$ cannot be a lattice.


Goncharov and Sorbi~\cite{GS-97} started developing the theory of \emph{generalized computable} numberings. One of their approaches to generalized computations can be summarized as follows. Let $\Gamma$ be a complexity class (e.g., $\Sigma^0_1$, $d$-$\Sigma^0_1$, $\Sigma^0_n$, or $\Pi^1_n$). A numbering $\nu$ of a family $\mathcal{S}$ is \emph{$\Gamma$-computable} if the set $G_{\nu}$ from (\ref{equ:g-nu}) belongs to the class $\Gamma$. We say that a family $\mathcal{S}$ is \emph{$\Gamma$-computable} if it has a $\Gamma$-computable numbering. Note that the classical notion of a computable numbering becomes a synonym for a $\Sigma^0_1$-computable numbering.

In a similar way to computable families, one can define the Rogers semilattice $\mathcal{R}_{\Gamma}(\mathcal{S})$ for a $\Gamma$-computable family $\mathcal{S}$, see Section~\ref{subsect:numb} for formal details.

We follow the approach of Goncharov and Sorbi, and study the following problem:
\begin{problem}\label{prob:01}
	Let $\Gamma$ be a class of the analytical hierarchy, i.e. $\Gamma \in \{ \Sigma^1_n, \Pi^1_n\,\colon$ $n\geq 1\}$. Study the elementary theories of Rogers semilattices for $\Gamma$-com\-pu\-ta\-ble families.
\end{problem}

The current paper is a continuation of studies developed in Refs.~\cite{Dor-16,BOY}. These papers concentrated on Rogers semilattices of $\Pi^1_n$-computable families. Dorzhieva~\cite{Dor-16} showed that for a $\Pi^1_n$-computable family $\mathcal{S}$, one of the following two conditions holds:
\begin{itemize}
	\item[(a)] either the Rogers semilattice $\mathcal{R}_{\Pi^1_n}(\mathcal{S})$ contains only one element,
	
	\item[(b)] or the first-order theory $Th(\mathcal{R}_{\Pi^1_n}(\mathcal{S}))$ is hereditarily undecidable.
\end{itemize}
The article~\cite{BOY} proves the following: if the semilattice $\mathcal{R}_{\Pi^1_n}(\mathcal{S})$ contains more than one element, then for any non-zero $m\neq n$ and any $\Pi^1_m$-com\-pu\-ta\-b\-le family $\mathcal{T}$, the structure $\mathcal{R}_{\Pi^1_n}(\mathcal{S})$ is not isomorphic to $\mathcal{R}_{\Pi^1_m}(\mathcal{T})$. Further related work is discussed in Section~\ref{sec:related}.

The paper~\cite{BOY} left the following problem open:
\begin{problem}\label{prob:02}
	Let $n$ be a non-zero natural number. Consider Rogers semilattices $\mathcal{R}_{\Pi^1_n}(\mathcal{S})$ for $\Pi^1_n$-com\-pu\-ta\-ble families $\mathcal{S}$. How many isomorphism types do these semilattices realize?
\end{problem}

While attacking Problem~\ref{prob:02}, we observed the following: In order to apply the known numbering-theoretic techniques in this setting, one \emph{needs} to employ additional set-theoretic assumptions.

This observation motivated us to organize our paper as follows: We prove a number of results concerning Problem~\ref{prob:01} under the assumption of \emph{Projective Determinacy} (\textbf{PD}, see Section~\ref{sect:PD} for the details). Why did we choose \textbf{PD} as an additional axiom? Tanaka~\cite{Tanaka} already initiated a systematic development of recursion theory for the levels of analytical hierarchy, under the assumption of \textbf{PD}. We found his approach well-suited to our goals.

In order to make our paper accessible to both numbering-theoretic and set-theoretic communities, we tried to make the exposition as self-contained as possible.

The structure of the paper is as follows. Section~\ref{sec:prelim} contains the necessary preliminaries on the theory of numberings. Section~\ref{sec:analyt} discusses the background on the analytical hierarchy. In particular, Section~\ref{sect:PD} introduces consequences of \textbf{PD}, which will be employed in our proofs. It is well-known that under \textbf{PD}, the levels of analytical hierarchy exhibit ``flip-flopping'' behavior: kindred levels are $\Pi^1_1$, $\Sigma^1_2$, $\Pi^1_3$, $\Sigma^1_4$, \dots (see, e.g., Refs.~\cite{Jech,Moschovakis}). Thus, following Refs.~\cite{Add-Mos-68,Tanaka}, we use the following conventions:
\begin{itemize}
	\item For a number $k\in\omega$, $E^1_{2k+1}:= \Pi^1_{2k+1}$ and $E^1_{2k+2}:= \Sigma^1_{2k+2}$.
	
	\item For a non-zero $n\in\omega$, we consider $E^1_n$-computable families $\mathcal{S}$. By $\mathcal{R}^1_n(\mathcal{S})$ we denote the Rogers semilattice $\mathcal{R}_{E^1_n}(\mathcal{S})$.
\end{itemize}

Section~\ref{sect:finite_fam} studies elementary properties of the semilattices $\mathcal{R}^1_n(\mathcal{S})$ for \emph{finite} families $\mathcal{S}$.
Any such $\mathcal{R}^1_n(\mathcal{S})$ is a distributive upper semilattice (Proposition~\ref{prop:distr}). Note that the proof of this fact does not require \textbf{PD}. While assuming \textbf{PD}, we obtain a criterion for when $\mathcal{R}^1_n(\mathcal{S})$ has the greatest element (Theorem~\ref{theo:greatest}). We also show that if $\mathcal{R}^1_n(\mathcal{S})$ has no greatest element, then it is upwards dense (Theorem~\ref{theo:min-cover}).

Section~\ref{sect:inf_fam} discusses first-order properties of $\mathcal{R}^1_n(\mathcal{S})$ for \emph{infinite} families $\mathcal{S}$. We make use of Dorzhieva's results from Ref.~\cite{Dor-16} to establish the following:
\begin{enumerate}
	\item For an infinite $E^1_n$-computable family $\mathcal{S}$, the upper semilattice $\mathcal{R}^1_n(\mathcal{S})$ is not weakly distributive (Theorem~\ref{theo:not-weak-distr}). Consequently, $\mathcal{R}^1_n(\mathcal{S})$ is not distributive.
	
	\item For an \emph{arbitrary} $E^1_n$-computable family $\mathcal{S}$ which contains at least two elements, the semilattice $\mathcal{R}^1_n(\mathcal{S})$ is infinite, and it is not a lattice (Corollary~\ref{cor:Dor}).
\end{enumerate}

Summarizing, our results (together with Lemma~\ref{lem:dual} below) provide a first step to the solution of Problem~\ref{prob:02}~--- now we know that under \textbf{PD}, there are at least four different isomorphism types of Rogers $E^1_n$-semilattices $\mathcal{R}^1_n(\mathcal{S})$, induced via the following families:
\begin{enumerate}
	\item A one-element family $\mathcal{S}$~--- in this case, the structure $\mathcal{R}^1_n(\mathcal{S})$ is also one-element.
	
	\item A finite family $\mathcal{S}$, containing more than one element and possessing the $\subseteq$-least element.
	
	\item A finite family $\mathcal{S}$ without the $\subseteq$-least element.
	
	\item An infinite $E^1_n$-computable family $\mathcal{S}$.
\end{enumerate}

The last section discusses further problems. In particular, we consider the following question: What happens to our results, if we replace \textbf{PD} with the Axiom of Constructibility ($V=L$)?

Recall that the \emph{Axiom of Dependent Choices} (\textbf{DC}) states the following: For any non-empty set $A$ and any set of pairs $P\subseteq A\times A$, we have:
\[
	(\forall x \in A)(\exists y\in A)P(x,y) \ \Rightarrow\ (\exists f\colon \omega \to A)(\forall n) P(f(n), f(n+1)).
\]
Throughout the paper, we work in set theory \textbf{ZF}$+$\textbf{DC}.


\section{Preliminaries} \label{sec:prelim}

Lower-case letters $x,y,z,\dots$ denote variables that range over $\omega$. Capital letters $X,Y,Z,\dots$ are used for subsets of $\omega$.

By $\leq_{\omega}$ we denote the standard ordering of natural numbers. Recall that $\omega^{\omega}$ is the set of all total functions acting from $\omega$ to $\omega$.

As per usual, $\langle \cdot,\cdot\rangle$ is a standard pairing function over $\omega$. By $\langle \cdot\rangle_0$ and $\langle \cdot\rangle_1$ we denote computable functions such that for every $n\in\omega$, we have $\langle \langle n\rangle_0, \langle n\rangle_1\rangle = n$.

We treat upper semilattices as structures in the language $L_{usl} = \{ \leq, \vee\}$.


\subsection{Numberings} \label{subsect:numb}

Suppose that $\nu$ is a numbering of a family $\mathcal{S}_0$, and $\mu$ is a numbering of a family $\mathcal{S}_1$. Notice that the condition $\nu \leq \mu$ always implies that $\mathcal{S}_0 \subseteq \mathcal{S}_1$.

Numberings $\nu$ and $\mu$ are \emph{equivalent}, denoted by $\nu\equiv \mu$, if $\nu\leq \mu$ and $\mu \leq \nu$. The numbering $\nu\oplus\mu$ of the family $\mathcal{S}_0\cup\mathcal{S}_1$ is defined as follows:
\[
	(\nu \oplus \mu)(2x) = \nu(x), \quad (\nu \oplus \mu)(2x+1) = \mu(x).
\]
The following fact is well-known (see, e.g., p.~36 in Ref.~\cite{Ershov-Book}): If $\xi$ is a numbering of a family $\mathcal{T}$, then
\[
	(\nu \leq  \xi \,\&\, \mu \leq  \xi) \ \Leftrightarrow\ (\nu\oplus \mu \leq  \xi).
\]
For further background on numberings, the reader is referred to, e.g., Refs.~\cite{Ershov-Book,Ershov-99,BG-00,Ershov-1,Ershov-2,Ershov-3}.

Let $\Gamma$ be a complexity class with the following properties:
\begin{itemize}
	\item[(a)] If $\nu$ is a $\Gamma$-computable numbering and $\mu$ is a numbering such that $\mu \leq \nu$, then $\mu$ is $\Gamma$-computable.
	
	\item[(b)] If numberings $\nu$ and $\mu$ are both $\Gamma$-computable, then the numbering $\nu\oplus \mu$ is also $\Gamma$-computable.
\end{itemize}
Note that it is not hard to show that for a non-zero natural number $n$, each of the classes $\Sigma^0_n$, $\Sigma^{-1}_n$, and $\Pi^1_n$ has these properties.

Let $\mathcal{S}$ be a $\Gamma$-computable family. By $Com_{\Gamma}(\mathcal{S})$ we denote the set of all $\Gamma$-computable numberings of $\mathcal{S}$. Since the relation $\equiv$ is a congruence on the structure $(Com_{\Gamma}(\mathcal{S}); \leq, \oplus)$, we use the same symbols $\leq$ and $\oplus$ on numberings and on their $\equiv$-equivalence classes.

The quotient structure $\mathcal{R}_{\Gamma}(\mathcal{S}) := (Com_{\Gamma}(\mathcal{S}) /_{\displaystyle{\equiv}}; \leq, \oplus)$ is an upper semilattice. We say that $\mathcal{R}_{\Gamma}(\mathcal{S})$ is the \emph{Rogers semilattice} of the $\Gamma$-computable family $\mathcal{S}$.

\subsection{Related work} \label{sec:related}

There is a large body of literature which studies a counterpart of Problem~\ref{prob:01} in the setting of the arithmetical hierarchy. For the sake of brevity, here we use the term \emph{Rogers $\Sigma^0_n$-semilattice} as a synonym for ``the Rogers semilattice of a $\Sigma^0_n$-computable family.''

Ershov and Lavrov~\cite{EL-73} (see also p.~72 in Ref.~\cite{Ershov-Book}, and Ref.~\cite{Ershov-03}) showed that there are finite families $\mathcal{S}_i$, $i\in\omega$, of c.e. sets such that the semilattices $\mathcal{R}_{\Sigma^0_1}(\mathcal{S}_i)$ are pairwise non-isomorphic. In other words, there are infinitely many isomorphism types of Rogers $\Sigma^0_1$-semilattices. V'yugin~\cite{V'yugin} proved that there are infinitely many pairwise elementarily non-equiva\-l\-ent Rogers $\Sigma^0_1$-semilattices. Badaev, Goncharov, and Sorbi~\cite{BGS-05} proved that for any natural number $n\geq 2$, there are infinitely many pairwise elementarily non-equivalent Rogers $\Sigma^0_n$-semilattices. The reader is referred to, e.g., Refs.~\cite{BGPS-03,BGS-06,BG-08,BMY,Podz} for further results on Rogers $\Sigma^0_n$-semilattices.

Recall that a numbering $\nu$ is \emph{Friedberg} if $\nu(k)\neq \nu(m)$ for all $k\neq m$. Dorzhieva~\cite{Dor-14,Dor-18} studied Friedberg numberings for families of sets in the analytical hierarchy.

Kalimullin, Puzarenko, and Faizrakhmanov~\cite{KPF-18,KPF-19} considered \emph{computable $\Pi^1_1$-numberings}. A \emph{$\Pi^1_1$-numbering} of a family $\mathcal{S}$ is a \emph{partial} map $\nu$ acting from $\omega$ onto $\mathcal{S}$ such that the domain of $\nu$ is enumeration reducible to the $\Pi^1_1$-complete set $\mathcal{O}$. A $\Pi^1_1$-numbering $\nu$ is \emph{computable} if the set
\[
	G^{\ast}_{\nu} = \{ \langle n,x\rangle \,\colon n\in dom(\nu),\ x\in \nu(n)\}
\]
is enumeration reducible to $\mathcal{O}$.


\section{Background on the analytical hierarchy} \label{sec:analyt}

Here we discuss known results on the analytical hierarchy, which will be employed in our proofs. Furthermore, the section includes proofs of several useful (but a little bit technical) results on numberings: Lemma~\ref{lem:dual}, Proposition~\ref{prop:univ-for-all}, and Lemma~\ref{lem:main-prop}.

Recall that a predicate $R(x_1,\dots,x_m;f_1,\dots,f_n)$, where the variables $f_i$ denote elements from $\omega^{\omega}$, is \emph{recursive} if there is an index $e\in\omega$ such that for all $f_1,\dots,f_n$ and $x_1,\dots,x_m$, the following conditions hold:
\begin{itemize}
	\item[(a)] the value $\varphi_e^{f_1\oplus \dots\oplus f_n}(x_1,\dots,x_m)$ is defined;
	
	\item[(b)] the predicate $R(x_1,\dots,x_m; f_1,\dots,f_n)$ is true if and only if 
	\[
		\varphi_e^{f_1\oplus \dots\oplus f_n}(x_1,\dots,x_m) = 1.
	\]
\end{itemize}
If an index $e$ satisfies these conditions, then we say that $e$ \emph{witnesses the recursiveness} of the predicate $R$.

We follow Ref.~\cite{Sacks} and use the following version of a \emph{normal form} for analytical subsets of $\omega$: Let $n$ be a non-zero natural number. A set $X\subseteq\omega^m$ is $\Pi^1_n$ if and only if there is a recursive predicate $R(x_1,\dots,x_m,y;f_1,\dots,f_n)$ such that for all $\bar a\in\omega^m$, we have
\begin{equation}\label{equ:normal_form}
	\bar a \in X\ \Leftrightarrow\ (\forall f_1) (\exists f_2) (\forall f_3) \dots (Qf_n) (\overline{Q}y) R(\bar a,y; f_1,\dots,f_n),
\end{equation}
where the last quantifiers are as follows:
\[
	Q = \begin{cases}
		\forall, & \text{if } n \text{ is odd},\\
		\exists, & \text{if } n \text{ is even};
	\end{cases}
	\quad
	\overline{Q} = \begin{cases}
		\exists, & \text{if } Q=\forall,\\
		\forall, & \text{if } Q=\exists.
	\end{cases}
\]

Mutatis mutandis, a $\Sigma^1_n$ set $X \subseteq \omega^m$ can be represented via the following form:
\[
	\bar a \in X\ \Leftrightarrow\ (\exists f_1) (\forall f_2) (\exists f_3) \dots (Qf_n) (\overline{Q}y) R(\bar a,y; f_1,\dots,f_n).
\]

Let $\Gamma \in \{ \Sigma^1_n, \Pi^1_n\,\colon n \geq 1\}$. By $\breve{\Gamma}$ we denote the dual class:
\[
	\breve{\Gamma} = \begin{cases}
		\Pi^1_n, & \text{if } \Gamma = \Sigma^1_n,\\
		\Sigma^1_n, & \text{if } \Gamma = \Pi^1_n.
	\end{cases}
\]

The next lemma will be useful for transferring various results from a class $\Gamma$ into its dual $\breve{\Gamma}$.

\begin{definition}
	Suppose that $\mathcal{S}$ is a countable family of subsets of $\omega$. By $Dual(\mathcal{S})$ we denote the following family:
	\[
		Dual(\mathcal{S}) := \{ A \,\colon\ \omega\setminus A \in \mathcal{S}\}.
	\]
\end{definition}

\begin{lemma}\label{lem:dual}
	Suppose that $\Gamma \in \{ \Sigma^1_n, \Pi^1_n\}$. Then the operator
	$\mathrm{Dual}\colon \mathcal{S} \mapsto Dual(\mathcal{S})$
	is a bijection from
	\begin{gather*}
		\{ \mathcal{S}\subset P(\omega)\,\colon \mathcal{S} \text{ is } \Gamma \text{-computable}\}\
		\text{onto } \{ \mathcal{T}\subset P(\omega)\,\colon \mathcal{T} \text{ is } \breve{\Gamma} \text{-computable}\}.
	\end{gather*}
	Furthermore, the semilattices $\mathcal{R}_{\Gamma}(\mathcal{S})$ and $\mathcal{R}_{\breve{\Gamma}}(Dual(\mathcal{S}))$ are isomorphic.
\end{lemma}
\begin{proof}[Proof Sketch]
	Given a $\Gamma$-computable numbering $\nu$ of a family $\mathcal{S}$, we define a numbering $\nu^{Dual}$ of the family $Dual(\mathcal{S})$ as follows: for $k\in\omega$,
	\[
		\nu^{Dual}(k) := \omega \setminus \nu(k).
	\]
	Clearly, the set $G_{\nu^{Dual}}$ is the complement of $G_{\nu}\in \Gamma$, and hence, the numbering $\nu^{Dual}$ is $\breve{\Gamma}$-computable. Furthermore, it is not hard to see that for arbitrary numberings $\nu$ and $\mu$, we have:
	\[
		\nu \leq \mu\ \Leftrightarrow\ \nu^{Dual} \leq \mu^{Dual}.
	\]
	This implies that the structures $\mathcal{R}_{\Gamma}(\mathcal{S})$ and $\mathcal{R}_{\breve{\Gamma}}(Dual(\mathcal{S}))$ are isomorphic.
	
	Note the following: if $\mathcal{S}\neq \mathcal{S}^{\ast}$, then $Dual(\mathcal{S}) \neq Dual(\mathcal{S}^{\ast})$. Moreover, $Dual(Dual(\mathcal{S}))=\mathcal{S}$. These facts are enough to finish the proof.
\end{proof}


\subsection{Universal numberings}

Here we give a brief discussion on universal numberings in the analytical hierarchy. We emphasize that the results of this subsection \emph{do not rely on} additional set-theoretic assumptions.

Let $n$ be a non-zero natural number, and let $\Gamma \in \{ \Pi^1_n, \Sigma^1_n\}$. We say that a $\Gamma$-computable numbering $\nu$ of a family $\mathcal{S}$ is \emph{universal} if $\nu$ induces the greatest element in the semilattice $\mathcal{R}_{\Gamma}(\mathcal{S})$.

\begin{proposition}[Kleene, see XIX in Ref.~\cite{Kleene-55}] \label{prop:univ-for-all}
	Let $\mathcal{S}$ be the family of all\ $\Gamma$ sets.
	There is a universal $\Gamma$-computable numbering of $\mathcal{S}$.
\end{proposition}
\begin{proof}
	Here we give a formal proof of the fact, so that a reader could get familiar with the techniques. We discuss only the case when $\Gamma = \Pi^1_n$. The proof for the $\Sigma^1_n$ case can be obtained in a similar way, mutatis mutandis.

	For a natural number $e$, set
	\[
		x\in \nu(e) \ \Leftrightarrow\ (\forall f_1) (\exists f_2) (\forall f_3) \dots (Qf_n) (\overline{Q} s) U(x,s;f_1,f_2,\dots,f_n),
	\]
	where a recursive predicate $U$ is defined as follows:
	\begin{itemize}
		\item If $n$ is odd, then $U(x,s; f_1,\dots,f_n)$ holds if and only if 
		\[
			\varphi^{f_1\oplus \dots \oplus f_n}_{e,\langle s\rangle_0}(x,\langle s\rangle_1) = 1.
		\]
		
		\item If $n$ is even, then $U(x,s; f_1,\dots,f_n)$ is true if and only if $\varphi^{f_1\oplus \dots \oplus f_n}_{e,\langle s\rangle_0}(x,\langle s\rangle_1)$ is either undefined or equal to one.
	\end{itemize}
	Clearly, the set $G_{\nu}$ is $\Pi^1_n$ and hence, the numbering $\nu$ is $\Pi^1_n$-computable.
	
	Let $\vec{F}$ denote an oracle $f_1\oplus \dots \oplus f_n$, where each $f_i$ belongs to $\omega^{\omega}$. Note the following key property of the relation $U$: for any oracle $\vec{F}$ and any number $x$, we have:
	\begin{itemize}
		\item If $n$ is odd, then the condition $\exists s U (x,s;f_1,\dots,f_n)$ is equivalent to $\exists y (\varphi^{\vec{F}}_{e}(x,y)\downarrow\ = 1)$.
		
		\item If $n$ is even, then $\forall s U (x,s;f_1,\dots,f_n)$ holds iff for every $y$, either $\varphi^{\vec{F}}_{e}(x,y)\uparrow$ or $\varphi^{\vec{F}}_{e}(x,y)\downarrow\ = 1$.
	\end{itemize}
	
	Let $X\subseteq \omega$ be an arbitrary $\Pi^1_n$ set. Choose a normal form from (\ref{equ:normal_form}) for the set $X$, and fix a number $i_0$ witnessing the recursiveness of $R$ from (\ref{equ:normal_form}). Then the key property of $U$ implies that $\nu(i_0) = X$. Hence, we deduce that $\nu$ is a numbering of the family of \emph{all} $\Pi^1_n$ sets.
	
	Now let $\mu$ be an arbitrary $\Pi^1_n$-computable numbering of $\mathcal{S}$. We need to show that $\mu \leq \nu$. Fix an index $j_0$ witnessing the recursiveness of the set $G_{\mu}\subseteq \omega$. Recall that this implies the following: $x\in  \mu(n)$ iff
	\[
		(\forall f_1) (\exists f_2) (\forall f_3) \dots (Qf_n) (\overline{Q} y) (\varphi^{\vec{F}}_{j_0}(\langle n,x\rangle,y ) = 1).
	\]
	By a relativized $s$-$m$-$n$ Theorem, there is a computable function $g(u,v)$ such that
	\[
		(\forall \vec{F})(\forall u) (\forall n) (\forall x) (\forall y) [ \varphi^{\vec{F}}_{u}(\langle n,x\rangle, y) =  \varphi^{\vec{F}}_{g(u,n)}(x, y)].
	\]
	Thus, the key property of the predicate $U$ implies that $\mu(n) = \nu(g(j_0,n))$ for all $n$. Hence, $\mu$ is reducible to $\nu$.
Proposition~\ref{prop:univ-for-all} is proved.
\end{proof}

Notice that in the proof above, we never use the fact that the numbering $\mu$ gives indices for \emph{all} $\Pi^1_n$ sets. Therefore, we deduce the following:
\begin{corollary}\label{coroll:univ}
	Let $\Gamma \in \{ \Pi^1_n, \Sigma^1_n\}$. There is a $\Gamma$-computable numbering $\nu_U$ such that an arbitrary $\Gamma$-computable numbering $\mu$ is reducible to $\nu_U$.
\end{corollary}


\subsection{Projective determinacy and its consequences} \label{sect:PD}

Tanaka~\cite{Tanaka} developed recursion theory for subsets of $\omega$, belonging to the levels of analytical hierarchy, under the assumption of \emph{Projective Determinacy} (\textbf{PD}). Here we follow Tanaka's approach: this subsection gives a brief overview of some consequences of \textbf{PD}, which will be heavily used in the proofs of our results.

First, we discuss the necessary set-theoretic background, our exposition mainly follows Refs.~\cite{Jech,Moschovakis}.

With each set $A \subseteq \omega^{\omega}$, one associates a two-person game $G = G(A)$ as follows. Players~I and~II alternatively choose natural numbers: player~I chooses $a_0$, then II chooses $b_0$, then I chooses $a_1$, then II chooses $b_1$, and so on. The game $G$ ends after $\omega$ steps. If the resulting sequence
\[
	f := (a_0,b_0,a_1,b_1,\dots)
\]
belongs to $A$, then player I wins. If $f\not\in A$, then II wins.

The game $G$ is a game of perfect information: before I choose $a_{n+1}$, she is allowed to see the tuple $(a_0,b_0,\dots, a_n,b_n)$; and similarly with II. A \emph{strategy} for the player I is a function $\sigma$, which maps tuples of even length to natural numbers. A strategy $\sigma$ is \emph{winning} if I always wins by following $\sigma$. In a similar way, one introduces the notion of a winning strategy for player II. The game $G(A)$ is \emph{determined} if one of the players has a winning strategy.

The \emph{Axiom of Determinacy} (\textbf{AD}) states that for every $A\subseteq \omega^{\omega}$, the game $G(A)$ is determined.

Now we recall the definition of \emph{projective hierarchy} for a Polish space $\mathcal{X}$. Recall that \emph{Baire space} is the set $\omega^{\omega}$, endowed with the natural topology.

A subset $A \subseteq \mathcal{X}$ is \emph{analytic} if either $A=\emptyset$, or there is a continuous map $f\colon \omega^{\omega} \to \mathcal{X}$ such that $range(f) = A$.

For a non-zero natural number $n$, the boldface classes $\mathbf{\Sigma}^1_n$ and $\mathbf{\Pi}^1_n$ (for the space $\mathcal{X}$) can be introduced as follows:
\begin{itemize}
	\item A set $A\subseteq \mathcal{X}$ is $\mathbf{\Sigma}^1_1$ if $A$ is analytic.
	
	\item $A$ is $\mathbf{\Pi}^1_1$ if its complement $\mathcal{X}\setminus A$ is analytic.
	
	\item $A$ is $\mathbf{\Sigma}^1_{n+1}$ if there is a $\mathbf{\Pi}^1_n$ set $B \subseteq \mathcal{X} \times \omega^{\omega}$ such that
	\[
		A = \{ a\,\colon (\exists f \in \omega^{\omega})[ (a,f) \in B ] \}.
	\]
	
	\item $A$ is $\mathbf{\Pi}^1_{n+1}$ if its complement is $\mathbf{\Sigma}^1_{n+1}$.
\end{itemize}
A set $A\subseteq \mathcal{X}$ is \emph{projective} if $A$ belongs to $\bigcup_{1\leq n < \omega} \mathbf{\Sigma}^1_n = \bigcup_{1\leq n < \omega} \mathbf{\Pi}^1_n$.

The axiom of \emph{Projective Determinacy} (\textbf{PD}) states that for every projective set $A\subseteq \omega^{\omega}$, the game $G(A)$ is determined.

For a detailed discussion of \textbf{AD} and \textbf{PD}, the reader is referred to, e.g., Refs.~\cite{Jech,Moschovakis}.

We follow Refs.~\cite{Add-Mos-68,Tanaka} and use the following notations: for a natural number $k$,
\begin{itemize}
	\item $E^1_{2k+1}$ (pronounced ``Epsilon$^1_{2k+1}$'') is the (lightface) class $\Pi^1_{2k+1}$, and $\Upsilon^1_{2k+1}$ is the class $\Sigma^1_{2k+1}$;
	
	\item $E^1_{2k+2} = \Sigma^1_{2k+2}$ and $\Upsilon^1_{2k+2} = \Pi^1_{2k+2}$.
\end{itemize}


\subsubsection{The prewellordering property} \label{subsub:prewell}

A binary relation $\preceq$ on a set $S$ is a \emph{prewellordering} if $\preceq$ is:
\begin{itemize}
	\item transitive;
	
	\item reflexive;
	
	\item connected, i.e. $x\preceq y$ or $y\preceq x$ for all $x,y\in S$;
	
	\item wellfounded, i.e. $S$ does not contain infinite descending chains $x_0 \succ x_1 \succ x_2 \succ \dots$.
\end{itemize}

A \emph{norm} on a set $S$ is an arbitrary function which maps $S$ into the ordinals. Given a norm $\varphi$ on $S$, one can associate with $\varphi$ the following prewellordering:
\[
	x \preceq^{\varphi} y \ \Leftrightarrow\ \varphi(x) \leq_{Ord} \varphi(y),
\]
where $\leq_{Ord}$ is the standard order on ordinals.

Suppose that $S\subseteq \omega$, and $\varphi\colon S \to \lambda$ is a norm on $S$. The map $\varphi$ is called a \emph{$\Gamma$-norm} if there are binary relations $\preceq^{\varphi}_{\Gamma}$ and $\preceq^{\varphi}_{\breve{\Gamma}}$ on $\omega$ such that $\preceq^{\varphi}_{\Gamma}$ belongs to $\Gamma$, $\preceq^{\varphi}_{\breve{\Gamma}}$ belongs to $\breve{\Gamma}$, and for every $y\in\omega$,
\begin{gather}
	\text{if } y \in S, \text{ then for any } x\in\omega, \notag\\
	[x\in S \,\&\, \varphi(x) \leq_{Ord} \varphi(y)] \ \Leftrightarrow\ x \preceq^{\varphi}_{\Gamma} y \ \Leftrightarrow\  x\preceq^{\varphi}_{\breve{\Gamma}} y. \label{def:Gamma-norm}
\end{gather}

A class $\Gamma$ has the \emph{prewellordering property} (or $\Gamma$ is \emph{normed}) if every set $S\in \Gamma$ admits a $\Gamma$-norm.

The following result is a consequence of the Prewellordering Theorem of Ref.~\cite{Add-Mos-68}. For a more detailed discussion, see Ref.~\cite{Mos-71} and Corollary 6B.2 of Ref.~\cite{Moschovakis}.

\begin{theorem}[Addison and Moschovakis~\cite{Add-Mos-68}] \label{theo:prewellord}
	Assume \textbf{PD}. Let $n$ be a non-zero natural number. The class $E^1_n$ has the prewellordering property.
\end{theorem}

\begin{remark}
	For the classes $\Pi^1_1$ and $\Sigma^1_2$, this result holds without assuming \textbf{PD}. See Theorems~XXIII and XXXVIII of Chap.~16 in Ref.~\cite{Rogers-book}; Theorems~4B.2 and 4B.3 of Ref.~\cite{Moschovakis}.
\end{remark}

Given an arbitrary $E^1_n$-computable numbering $\nu\colon \omega \to P(\omega)$, we define a relation $\sqsubseteq^{\nu}\ \subseteq \omega^2 \times \omega^2$ as follows.

Since the set $G_{\nu} = \{ \langle k,x\rangle\,\colon x\in\nu(k)\}$ is $E^1_n$, by Theorem~\ref{theo:prewellord}, one can choose a $E^1_n$-norm $\varphi$ mapping $G_{\nu}$ into some ordinal $\lambda$. Fix binary relations $\preceq^{\varphi}_{\Gamma}$ and $\preceq^{\varphi}_{\breve{\Gamma}}$, where $\Gamma = E^1_n$, witnessing that $\varphi$ is a $E^1_n$-norm.

For natural numbers $k,m,x,y$, we say that $(k,x) \sqsubseteq^{\nu} (m,y)$ if
\begin{gather*}
	(k=m) \ \&\  [ (\langle k,x\rangle \preceq^{\varphi}_{\Gamma} \langle m,y\rangle  \,\&\, \langle m,y\rangle \not\preceq^{\varphi}_{\Gamma} \langle k,x\rangle)  \ \vee\\
	 (\langle k,x\rangle \preceq^{\varphi}_{\Gamma} \langle m,y\rangle  \,\&\, \langle m,y\rangle \preceq^{\varphi}_{\Gamma} \langle k,x\rangle \,\&\, x\leq_{\omega}y )].
\end{gather*}

Informally speaking, the result below introduces the ``building blocks'' of our constructions: The sets $\widehat{[x]}_{\nu(k)}$ allow us to transfer some results (already known for, say, $\Sigma^0_2$-computable numberings) into our setting.

\begin{lemma}[Main Property of $\sqsubseteq^{\nu}$] \label{lem:main-prop}
	Assume \textbf{PD}. Suppose that $\nu$ is a $E^1_n$-com\-pu\-ta\-ble numbering, and $k\in\omega$. Then:
	\begin{itemize}
		\item[(i)] The relation $\preceq^{\nu(k)}\ := \{ (x,y)\,\colon x,y \in \nu(k),\ (k,x) \sqsubseteq^{\nu} (k,y)\}$ is a wellordering on the set $\nu(k)$.
	
	 	\item[(ii)] For a number $x\in\nu(k)$, the set
		\[
			\widehat{[x]}_{\nu(k)} := \{  z\in \omega \,\colon (k,z) \sqsubseteq^{\nu} (k,x)\}
		\]
		is a $\Delta^1_n$ subset of $\nu(k)$. Furthermore, the formulas witnessing the $\Delta^1_n$-ness do not depend on the choice of $k$ and $x$.
	\end{itemize}
\end{lemma}
\begin{proof}
	(i)\ The definition of the relation $\sqsubseteq^{\nu}$ implies that for arbitrary numbers $x,y$ from $\nu(k)$, the condition $x \preceq^{\nu(k)} y$ holds iff either $\varphi(\langle k,x \rangle) \lneq_{Ord} \varphi( \langle k,y\rangle)$, or $\varphi(\langle k,x\rangle) = \varphi(\langle k,y\rangle)\ \&\ x\leq_{\omega} y$. Hence, the map
	\[
		\psi \colon x \mapsto (x,\varphi(\langle k,x\rangle))
	\]
	induces an isomorphic embedding from $\nu(k)$ into the ordinal $\omega\cdot \lambda$. Therefore, we deduce that the poset $(\nu(k); \preceq^{\nu(k)})$ is well-ordered.

	(ii)\ Since $\langle k,x\rangle \in G_{\nu}$, for an arbitrary $z \in \widehat{[x]}_{\nu(k)}$, we have $\langle k,z\rangle \preceq^{\varphi}_{\Gamma} \langle k,x\rangle$ and hence, by the definition of a $\Gamma$-norm, $\langle k,z\rangle \in G_{\nu}$. In other words, we showed that $\widehat{[x]}_{\nu(k)} \subseteq \nu(k)$.
	
	Condition~(\ref{def:Gamma-norm}) implies that the following conditions are equivalent:
	\begin{itemize}
		\item[(a)] $z \preceq^{\nu(k)} x$;
		
		\item[(b)] ($\langle k,z\rangle \preceq^{\varphi}_{\Gamma} \langle k,x\rangle$ and $\langle k,x\rangle \not\preceq^{\varphi}_{\breve{\Gamma}} \langle k,z\rangle$), or ($\langle k,z\rangle \preceq^{\varphi}_{\Gamma} \langle k,x\rangle$ and $\langle k,x\rangle \preceq^{\varphi}_{\Gamma} \langle k,z\rangle$ and $z\leq_{\omega}x$);
		
		\item[(c)] ($\langle k,z\rangle \preceq^{\varphi}_{\breve{\Gamma}} \langle k,x\rangle$ and $\langle k,x\rangle \not\preceq^{\varphi}_{\Gamma} \langle k,z\rangle$), or ($\langle k,z\rangle \preceq^{\varphi}_{\breve{\Gamma}} \langle k,x\rangle$ and $\langle k,x\rangle \preceq^{\varphi}_{\breve{\Gamma}} \langle k,z\rangle$ and $z\leq_{\omega}x$).
	\end{itemize}
	Condition~(b) is logically equivalent to a $E^1_n$ formula, and Condition~(c) is equivalent to a $\Upsilon^1_n$ formula. Clearly, the formulas do not depend on the choice of $k$ and $x$~--- they only depend on the choice of the relations $\preceq^{\varphi}_{\Gamma}$ and $\preceq^{\varphi}_{\breve{\Gamma}}$. Lemma~\ref{lem:main-prop} is proved.
\end{proof}


\subsubsection{Reduction principle}

Suppose that $\Gamma \in \{ \Sigma^1_n, \Pi^1_n\,\colon n\geq 1\}$. The class $\Gamma$ satisfies the \emph{reduction principle} if for every pair $A,B\in \Gamma$, there are $A^{\ast}, B^{\ast}\in \Gamma$ such that:
\begin{gather*}
	A^{\ast} \subseteq A,\quad B^{\ast} \subseteq B,\\
	A^{\ast} \cup B^{\ast} = A\cup B,\quad A^{\ast} \cap B^{\ast} = \emptyset.
\end{gather*}

\begin{theorem}[Addison and Moschovakis~\cite{Add-Mos-68}; Martin~\cite{Martin}]\label{theo:reduction}
	Assume~\textbf{PD}. Let $n$ be a non-zero natural number. The class $E^1_n$ satisfies the reduction principle.
\end{theorem}

\begin{remark}
	The classes $\Pi^1_1$ and $\Sigma^1_2$ satisfy reduction principle, without assuming \textbf{PD} (Addison~\cite{Add-59}).
\end{remark}

Note that the reduction principle for $E^1_n$ (as formulated above) follows from Theorem~\ref{theo:prewellord}, see Exercise~4B.10 and Corollary~6B.2 of Ref.~\cite{Moschovakis}.



\section{Rogers semilattices for finite families} \label{sect:finite_fam}

For the sake of convenience, from now on, we will use the following notation:
\[
	\mathcal{R}^1_n(\mathcal{S}) := \mathcal{R}_{E^1_n}(\mathcal{S}).
\]

The section discusses elementary properties of the semilattices $\mathcal{R}^1_n(\mathcal{S})$, where $\mathcal{S}$ is finite. As a warm-up, we apply standard numbering-theoretic techniques (see, e.g., Ref.~\cite{BGS-03-isom-3}) to prove that $\mathcal{R}^1_n(\mathcal{S})$ is a distributive upper semilattice (Proposition~\ref{prop:distr}). We also show that if $\mathcal{S}$ contains more than one element, then $\mathcal{R}^1_n(\mathcal{S})$ is infinite and cannot be a lattice (Proposition~\ref{prop:inf-not-lat}). Note that these proofs \emph{do not require} additional set-theoretic assumptions. After that, we obtain a criterion for when $\mathcal{R}^1_n(\mathcal{S})$ has the greatest element (Theorem~\ref{theo:greatest}). The proof of this criterion already heavily employs the consequences of \textbf{PD} discussed in Section~\ref{sect:PD}. After that, the developed techniques are used to prove the following: If $\mathcal{R}^1_n(\mathcal{S})$ has no greatest element, then it is upwards dense (Theorem~\ref{theo:min-cover}).


Before proceeding to main results, we observe the following very simple fact:
\begin{remark}
	Let $\mathcal{S}$ be a non-empty finite family of $E^1_n$ sets. Then $\mathcal{S}$ is $E^1_n$-computable, and the semilattice $\mathcal{R}^1_n(\mathcal{S})$ contains the least element.
\end{remark}
\begin{proof}
 Suppose that $\mathcal{S} = \{B_0,B_1,\dots,B_m\}$. Then the numbering
\[
	\mu(k) := \begin{cases}
		B_k, & \text{if } k\leq m,\\
		B_0, & \text{otherwise},
	\end{cases}
\]
is a $E^1_n$-computable numbering of $\mathcal{S}$.

Assume that $\nu$ is an arbitrary $E^1_n$-computable numbering of $\mathcal{S}$. Choose indices $b_i$, $i\leq m$, with $\nu(b_i) = B_i$. Clearly, a computable function
\[
	f(k) := \begin{cases}
		b_k, & \text{if } k\leq m,\\
		b_0, & \text{otherwise},
	\end{cases}
\]
provides a reduction from $\mu$ into $\nu$.
\end{proof}


Recall that an upper semilattice $\mathcal{A} = (A;\leq, \vee)$ is \emph{distributive} if for all $b,a_0,a_1\in\mathcal{A}$, the following holds:
\[
	(b\leq a_0 \vee a_1) \ \Rightarrow\ \exists b_0 \exists b_1 [ (b = b_0 \vee b_1) \,\&\, (b_0 \leq a_0) \,\&\,  (b_1 \leq a_1)].
\]

The next two propositions establish general facts about Rogers $E^1_n$-se\-mi\-lat\-ti\-ces of finite families, note that these facts do not depend on \textbf{PD}.

\begin{proposition} \label{prop:distr}
	Suppose that $\mathcal{S} = \{ A_0, A_1,\dots,A_m\}$ is a finite family of $E^1_n$ sets. Then the upper semilattice $\mathcal{R}^1_n(\mathcal{S})$ is distributive.
\end{proposition}
\begin{proof}
	Let $\mu$, $\nu_0$, and $\nu_1$ be $E^1_n$-computable numberings of $\mathcal{S}$. Suppose that a computable function $f(x)$ reduces $\mu$ to $\nu_0 \oplus \nu_1$. Consider computable sets
	\[
		R_0 = \{ x\,\colon f(x) \text{ is even}\},\quad
		R_1 = \{ x\,\colon f(x) \text{ is odd}\}.
	\]
	
	First, assume that one of these sets, say, $R_1$ is empty. Then $\mu$ is reducible to $\nu_0$, and one can define $\mu_0 := \mu$, and
	\[
		\mu_1(k) := \begin{cases}
			A_k, & \text{if } k \leq m,\\
			A_0, & \text{if } k > m.
		\end{cases}
	\]
	It is not difficult to see that $\mu \equiv \mu_0 \oplus \mu_1$, $\mu_0 \leq \nu_0$, and $\mu_1 \leq \nu_1$.
	
	Therefore, from now on, we may assume that both $R_0$ and $R_1$ are non-empty. For $i\in\{ 0,1\}$, fix a total computable function $g_i$ with $range(g_i) = R_i$. Define a $E^1_n$-computable numbering $\xi_i:= \mu\circ g_i$. It is not hard to show (see, e.g., Proposition~3.1 of Ref.~\cite{BGS-03-isom-3}) that $\xi_i\leq \nu_i$ and $\mu \equiv \xi_0 \oplus \xi_1$. Note that in general, $\xi_i$ indexes \emph{not all} elements of $\mathcal{S}$.
	
	For $i\in \{0,1\}$, we define
	\[
		\mu_i(k) := \begin{cases}
			A_k, & \text{if } k\leq m,\\
			\xi_i(k-m-1), & \text{if } k > m.
		\end{cases}
	\]
	Clearly, each $\mu_i$ is a $E^1_n$-computable numbering of the family $\mathcal{S}$. On the other hand, it is straightforward to show that $\mu_i \leq \nu_i$ and $\mu \equiv \mu_0\oplus \mu_1$. Therefore, the semilattice $\mathcal{R}^1_n(\mathcal{S})$ is distributive. This concludes the proof of Proposition~\ref{prop:distr}.
\end{proof}

For an $E^1_n$ set $A$, consider a one-element family $\mathcal{T}:=\{ A\}$. It is obvious that the family $\mathcal{T}$ has \emph{only one} numbering, and hence, the semilattice $\mathcal{R}^1_n(\mathcal{T})$ is one-element.

The next proposition shows that if a finite family $\mathcal{S}$ contains more than one element, then for this $\mathcal{S}$, one can recast the results of Khutoretskii~\cite{Khut-71} and Selivanov~\cite{Sel-76}, mentioned in the introduction.

\begin{proposition}\label{prop:inf-not-lat}
	Let $\mathcal{S} = \{A_0, A_1, \dots, A_m\}$ be a finite family of $E^1_n$ sets, which contains at least two elements. Then the Rogers semilattice $\mathcal{R}^1_n(\mathcal{S})$ is infinite, and it is not a lattice.
\end{proposition}
\begin{proof}
	We follow the proof of Theorem~2.1 in Ref.~\cite{GS-97}. Note that $m\geq 1$. Let $W$ be a computably enumerable set such that $W\neq \emptyset$ and $W\neq \omega$. We define a numbering $\nu^{W}$ as follows: for $k\in\omega$, set
	\[
		\nu^{W}(k) := \begin{cases}
			A_k, & \text{if } k\leq m-2,\\
			A_{m-1}, & \text{if } k\geq m-1 \text{ and } (k-m+1) \in W,\\
			A_m, & \text{if } k\geq m-1 \text{ and } (k-m+1) \not\in W.
		\end{cases}
	\]
	It is clear that $\nu^{W}$ is a $E^1_n$-computable numbering of the family $\mathcal{S}$.
	
	Let $\mu$ be a numbering of $\mathcal{S}$ such that $\mu\leq \nu^{W}$. Fix a computable function $f$, which reduces $\mu$ to $\nu^{W}$. Then the set $V:=\{ k\,\colon f(k)\in W\}$ is c.e., and it is straightforward to show that $V\leq_m W$ and $\mu\equiv \nu^{V}$. Consider a principal ideal $\mathcal{I}$, induced inside $\mathcal{R}^1_n(\mathcal{S})$ by the (degree of the) numbering $\nu^{\emptyset'}$. One can show that $\mathcal{I}$ is isomorphic to the upper semilattice of c.e. $m$-degrees $\mathbf{R}_m$. Since $\mathbf{R}_m$ is not a lattice, the structure $\mathcal{R}^1_n(\mathcal{S})$ also cannot be a lattice. Proposition~\ref{prop:inf-not-lat} is proved.
\end{proof}


\subsection{Existence of a greatest element}

Now we proceed to investigating the following question (while assuming \textbf{PD}): When does the semilattice $\mathcal{R}^1_n(\mathcal{S})$ has the greatest element?

\begin{theorem}[\textbf{PD}] \label{theo:greatest}
	Let $n$ be a non-zero natural number.
	Let \[
		\mathcal{S} = \{ A_1, A_2,\dots, A_m\}
	\] 
	be a finite family of $E^1_n$ sets.
	Then the Rogers semilattice $\mathcal{R}^1_n(\mathcal{S})$ has the greatest element if and only if the family $\mathcal{S}$ contains a least element under set-theoretic inclusion.
\end{theorem}
\begin{proof}
	The basic idea is essentially the same as that of Theorem~3.2 in Ref.~\cite{BGS-03-completeness}, but we also have to carefully interweave set-theoretic results from Section~\ref{sect:PD}.

	Choose a family of finite sets $\{F_1,F_2,\dots,F_m\}$ with the following property: for all $i$ and $j$, we have
	\[
		F_i \subseteq F_j\ \Leftrightarrow\ A_i \subseteq A_j \ \Leftrightarrow\ F_i \subseteq A_j.
	\]

	We fix a universal $E^1_n$-computable numbering $\nu_U$ from Corollary~\ref{coroll:univ}. For the sake of brevity, the relation $\sqsubseteq^{\nu^{U}}$ from Lemma~\ref{lem:main-prop} will be denoted by $\sqsubseteq$.

	\underline{($\Leftarrow$).} Suppose that $A_1$ is the least element inside $\mathcal{S}$. Without loss of generality, we may assume that $F_1=\emptyset$.
	
	Before giving a formal construction, we outline the intuition behind our proof. For starters, we sketch a \emph{classical} argument: Assume that a family $\mathcal{T}$ contains precisely the following sets:
	\[
		\emptyset, \{0\}, \{1\}, \{0,1\}, \{1,2\},
	\]
	and we want to build a universal $\Sigma^0_1$-computable numbering of $\mathcal{T}$.
	
	In order to achieve this, first, we choose a universal $\Sigma^0_1$-computable numbering $\nu$ of the family of \emph{all} c.e. sets. Note that, similarly to Corollary~\ref{coroll:univ}, \emph{any} $\Sigma^0_1$-computable numbering (of any family) is reducible to $\nu$.
	
	 We describe an effective procedure that transforms the numbering $\nu$ into a $\Sigma^0_1$-computable numbering $\xi$, which indexes \emph{only} elements from $\mathcal{T}$. Fix an effective uniform approximation
	 \[
	 	\nu(k) = \bigcup_{s\in\omega} \nu^s(k),
	 \]
	 where every $\nu^s(k)$ is a finite set, and $\nu^s(k) \subseteq \nu^{s+1}(k)$. As per usual, we may assume that the difference $\nu^{s+1}(k)\setminus \nu^s(k)$ contains at most one element.
	
	The desired procedure works in a pretty straightforward, ``dynamic'' fashion:
	\begin{itemize}
		\item[(a)] While neither $0$, nor $1$ belongs to $\nu^s(k)$, we just put $\xi^s(k) := \emptyset$. Assume that $s_0$ is the least step such that (precisely) one of the elements $0$ or $1$ appears in $\nu^{s_0}(k)$.
		
		\item[(b.1)] If $0 \in \nu^{s_0}(k)$, then define $\xi^{s_0}(k):= \{0\}$. Wait for a first step $s_1 > s_0$ with $1\in \nu^{s_1}(k)$. While waiting, do not change the (approximation of the) set $\xi(k)$. If such a step appears, then set $\xi(k) = \xi^{s_1}(k) := \{ 0,1\}$.
		
		\item[(b.2)] Otherwise, we have $1\in \nu^{s_0}(k)$. Put $\xi^{s_0}(k) := \{1\}$. Wait for a first step $s_2 > s_0$ such that one of the elements $0$ or $2$ belongs to $\nu^{s_2}(k)$. While waiting, don't change $\xi(k)$. When such a step $s_2$ is found, denote the element from $\{0,2\}\cap \nu^{s_2}(k)$ as $v$, and define $\xi(k) = \xi^{s_2}(k) := \{ 0,v\}$.
	\end{itemize}
	Clearly, the described procedure is uniform in $k$. Moreover, it builds a $\Sigma^0_1$-co\-m\-pu\-ta\-ble numbering $\xi$, and for each $k$, the constructed $\xi(k)$ is an element from $\mathcal{T}$. The key property of the construction is provided by the following simple observation: if $\nu(k)$ is already an element of $\mathcal{T}$, then $\xi(k)$ is \emph{equal} to $\nu(k)$.
	
	Now let $\mu$ be an arbitrary $\Sigma^0_1$-computable numbering of $\mathcal{T}$. Since the numbering $\nu$ is universal, one can choose a computable function $f(x)$ such that $\mu(k) = \nu(f(k))$ for all numbers $k$. The key property of the construction implies that $\mu(k) = \nu(f(k))=\xi(f(k))$. In other words, $\mu \leq \xi$, and $\xi$ induces the greatest element of the Rogers semilattice for the $\Sigma^0_1$-computable family $\mathcal{T}$.
	
	This concludes the description of the classical argument. The argument \emph{cannot} be readily transferred to the $E^1_n$ setting: roughly speaking, we do not know what is an effective approximation of a given $E^1_n$-com\-pu\-ta\-ble numbering $\nu$. Nevertheless, this obstacle can be circumvented as follows~--- a careful analysis of the construction above reveals that the \emph{only} important question we need to address is the following:
	\begin{gather*}
		\text{What does one mean, when she says: ``The number 0 appears}\\
		\text{\emph{earlier} than the number 1, in an approximation of $\nu(k)$''?}
	\end{gather*}
	And we \emph{can} give a precise answer to this question: Zero appears earlier than one iff there is a number $x\in \nu(k)$ such that $0\in \widehat{[x]}_{\nu(k)}$ and $1\not\in \widehat{[x]}_{\nu(k)}$. In other words, the wellorder $(\nu(k); \preceq^{\nu(k)})$ provided by Lemma~\ref{lem:main-prop} allows us to ``replace'' the stages $s_0,s_1,s_2$, used in the classical construction, by elements $x_0,x_1,x_2$ belonging to the set $\nu(k)$ itself. The formal details of this (quite informal) idea are elaborated below.
	
	
	Consider a set $I := \{ 2,3,\dots,m\}$, note that here we do not include $1$ into $I$. We introduce a partial order $\unlhd$ on $I$ as follows. For $i,j\in I$, we say that $i \unlhd j$ iff $A_i\subseteq A_j$.
	
	For a number $i\in I$, its upper cone is the following set:
	\[
		\widetilde{(i)} := \{ j\in I\,\colon j \unrhd i,\  j\neq i\}.
	\]
	Notice that $i$ itself does not belong to the cone $\widetilde{(i)}$.

	Let $C = \{ i_0 \lhd i_1 \lhd \dots \lhd i_t\}$ be a maximal (under set-inclusion) increasing chain inside the poset $(I;\unlhd)$. We define the following formulas:
	\begin{gather*}
		\psi^C_{0}(k) := \exists x \bigg[ x\in \nu^U(k) \,\&\, \bigwedge_{y\in F_{i_0}} (k,y) \sqsubseteq (k,x) \,\&\, \\
		\bigwedge_{j \text{ is minimal inside } I;\, j\neq i_0}\ \bigvee_{z \in F_j} (k,z) \not\sqsubseteq (k,x) \bigg];\\
		\psi^C_{l+1}(k) := \psi^{C}_{l}(k) \,\&\, \exists x \bigg[ x\in \nu^U(k) \,\&\, \bigwedge_{y\in F_{i_{l+1}}} (k,y) \sqsubseteq (k,x) \,\&\, \\
		\bigwedge_{j \text{ is minimal inside } \widetilde{(i_l)};\, j\neq i_{l+1}}\ \bigvee_{z \in F_j} (k,z) \not\sqsubseteq (k,x) \bigg];\\
		\psi^C(x,k) := \bigvee_{l\leq t} [ \psi_{l}^C(k) \,\&\, x\in A_{i_l} ].
	\end{gather*}
	The main property of the relation $\sqsubseteq$ (Lemma~\ref{lem:main-prop}) implies that each of these formulas is logically equivalent to a $E^1_n$ condition.
	
	We define a new numbering $\xi$ as follows: $x\in \xi(k)$ iff
	\[
		(x\in A_1) \text{ or } \bigvee_{C \text{ is a maximal chain in } I} \psi^C(x,k).
	\]
	Clearly, the set $G_{\xi}$ belongs to $E^1_n$ and thus, the numbering $\xi$ is $E^1_n$-com\-pu\-ta\-ble.  We establish the following important property of $\xi$:
	
	\begin{claim}\label{claim:first}
		For any $k$, the set $\xi(k)$ belongs to the family $\mathcal{S}$. Moreover, if $\nu^U(k)\in\mathcal{S}$, then $\xi(k) = \nu^{U}(k)$.
	\end{claim}
	\begin{proof}
		Let $k$ be an arbitrary natural number. Without loss of generality, we may assume that $\xi(k)\neq A_1$. It is sufficient to prove the following fact: There is a maximal chain $C = \{ i_0 \lhd i_1 \lhd \dots \lhd i_t\}$ and a number $r\leq t$ such that $\xi(k) = \bigcup_{l\leq r} A_{i_l} = A_{i_r}$. We build the desired chain $C$ step-by-step.
		
		Since $\xi(k) \neq A_1$, there is a chain $D_0$ inside $I$ such that $\psi_0^{D_0}(k)$ holds. We define $i_0$ as the least element of  $D_0$. Recall that Lemma~\ref{lem:main-prop} shows the following: if $x\in \nu^U(k)$ and $(k,y) \sqsubseteq (k,x)$, then $y$ also belongs to $\nu^U(k)$. Thus, we deduce that $F_{i_0} \subseteq \nu^U(k)$ and $A_{i_0} \subseteq \xi(k)$.
		
		The relation $\preceq^{\nu(k)}$ is a wellordering on $\nu(k)$. This implies that there is the $\preceq^{\nu(k)}$-least number $x_0\in \nu(k)$ such that $F_j \subseteq \widehat{[x_0]}_{\nu(k)}$ for some $\unlhd$-mi\-ni\-mal $j\in I$. Since $\psi_0^{D_0}(k)$ is true, this particular $\unlhd$-minimal $j$ must be equal to $i_0$. From this, we deduce the following: if a maximal chain $D'$ starts with an element $j'$ not equal to $i_0$, then $\psi^{D'}(x,k)$ is false for all $x$. Such chains $D'$ can be omitted from further consideration.
		
		Now assume that $i_l$ has been already defined, and we know the following:
		\begin{itemize}
			\item[(a)] $i_0 \lhd i_1 \lhd \dots \lhd i_l$;
			
			\item[(b)] $F_{i_l} \subseteq \nu^{U}(k)$ and $A_{i_l} \subseteq \xi(k)$;
			
			\item[(c)] For each $u\leq l$, if a maximal chain $D'$ does not start with $i_0 \lhd i_1 \lhd \dots \lhd i_u$, then $\psi_{q}^{D'}(k)$ is false for all $q\geq u$.
		\end{itemize}
		Consider the following cases:
		
		\emph{Case 1.} Assume that $C := \{ i_0 \lhd i_1 \lhd \dots \lhd i_l \}$ is already a maximal chain inside $I$. Then the items~(b) and~(c) together imply that $\xi(k) = \bigcup_{u\leq l} A_{i_u} = A_{i_l}$, and this finishes the construction.
		
		\emph{Case 2.} Suppose that $i_0 \lhd \dots \lhd i_l$ is not a maximal chain.
		
		\emph{Case 2.1.} Assume that for every $j \rhd i_l$, we have $F_{j} \not\subseteq \nu^U(k)$. Then Lemma~\ref{lem:main-prop} implies that for any maximal chain $D$ starting with $i_0 \lhd \dots \lhd i_l$, the formula $\psi^D_{l+1}(k)$ is false. Hence, one obtains that $\xi(k) = A_{i_l}$, and the desired $C$ can be chosen as an arbitrary maximal chain beginning with $i_0 \lhd \dots \lhd i_l$.
		
		\emph{Case 2.2.} Suppose that there is $j \rhd i_l$ with $F_{j} \subseteq \nu^U(k)$. Since $\preceq^{\nu(k)}$ is a wellordering, there is the $\preceq^{\nu(k)}$-least number $x_{l+1}\in \nu(k)$ such that $F_{j^{\ast}} \subseteq \widehat{[x_{l+1}]}_{\nu(k)}$ for some $\unlhd$-minimal $j^{\ast}$ from $\widetilde{(i_l)}$. Furthermore, this number $j^{\ast}$ is uniquely determined, and for any maximal chain $D$ starting with $i_0 \lhd \dots \lhd i_l \lhd j^{\ast}$, the formula $\psi^D_{l+1}(k)$ is true. We put $i_{l+1} := j^{\ast}$ and proceed to finding $i_{l+2}$. Clearly, the following holds:
		\begin{itemize}
			\item $i_{l+1}\rhd i_l$;
			
			\item $F_{i_{l+1}} \subseteq \nu^{U}(k)$ and $A_{i_{l+1}} \subseteq \xi(k)$;
			
			\item For every maximal chain $D'$ not beginning with $i_0 \lhd \dots \lhd i_l\lhd i_{l+1}$, the formulas $\psi^{D'}_{q}$, $q\geq l+1$, are false.
		\end{itemize}
		
		Since the poset $I$ is finite, the described construction finishes after finitely many iterations, and produces the desired maximal chain $C$ and number $i_r$ such that $\xi(k) = A_{i_r}$. In particular, every $\xi(k)$ belongs to $\mathcal{S}$.
		
		Now assume that $\nu^U(k)=A_j$ for some $j\leq m$. Clearly, if $j=1$, then $\xi(k)$ is also equal to $A_1$. Hence, we assume that $j\geq 2$. In this particular case, the construction above finishes only when one of the following two conditions is satisfied:
		\begin{enumerate}
			\item We found a $\unlhd$-maximal $i_l$ with $F_{i_l}\subseteq \nu^{U}(k)$. Then, clearly, $j=i_l$ and $\xi(k) = \nu^U(k)$.
			
			\item We found a number $i_l$ such that $F_{i_l} \subseteq \nu^U(k)$, but for every $q\rhd i_l$, the set $F_{q}$ is not a subset of $\nu^U(k)$. We deduce that $A_{i_l} \subseteq A_j$, and for any $q$, the condition $A_{q} \supsetneq A_{i_l}$ implies $A_{q} \neq A_j$. Hence, $j=i_l$ and $\xi(k) = A_{i_l} = \nu^{U}(k)$.
		\end{enumerate}
		Claim~\ref{claim:first} is proved.
	\end{proof}
	
	Recall that the numbering $\nu^U$ is universal. Thus, for any $E^1_n$-com\-pu\-ta\-ble numbering $\mu$ of the family $\mathcal{S}$, there is a computable function $f(x)$ such that $\mu(k) = \nu^U(f(k))$ for all $k$. Claim~\ref{claim:first} implies that $\mu(k) = \nu^U(f(k)) = \xi(f(k))$, and therefore, $\xi$ is a universal $E^1_n$-computable numbering of the family $\mathcal{S}$.
	
	
	\
	
	\underline{($\Rightarrow$).} Suppose that the family $\mathcal{S}$ has no least element under $\subseteq$. Without loss of generality, we may assume that $A_1,A_2,\dots,A_l$ (where $2\leq l\leq m$) are all $\subseteq$-minimal elements of $\mathcal{S}$.
	
	In order to prove the desired fact, it is sufficient to obtain the following: Given an arbitrary $E^1_n$-computable numbering $\nu$ of the family $\mathcal{S}$, one can construct a $E^1_n$-computable numbering $\mu$ of $\mathcal{S}$ such that $\mu\nleq \nu$.
	
	Fix indices $p_i$, $1\leq i\leq l$, such that $\nu(p_i) = A_i$. Let
	\[
		p_{\sigma(i)} := \begin{cases}
			p_{i+1}, & \text{if } i<l,\\
			p_1, & \text{if } i = l.
		\end{cases}
	\]

	For each non-zero $i\leq l$, consider a $E^1_n$ set
	\[
		Q_i := \{ k\in\omega \,\colon F_i \subseteq \nu(k)\}.
	\]
	Clearly, $\bigcup_{i\leq l} Q_i$ is equal to $\omega$ (recall that we picked \emph{all} $\subseteq$-minimal elements of $\mathcal{S}$).	By the reduction principle (Theorem~\ref{theo:reduction}), there are pairwise disjoint $E^1_n$ sets $R_i$, $i\leq l$, such that
	\[
		R_i \subseteq Q_i \text{ and } \bigcup_{i\leq l} R_i = \omega.
	\]
	Notice that in particular, every $R_i$ is a $\Delta^1_n$ set.
	
	We define a total function $f\colon\omega\to \{ p_i\,\colon 1\leq i\leq l\}$. Set
	\[
		f(k) := \begin{cases}
			p_{\sigma(i)}, & \text{if } \varphi_k(k)\!\downarrow \text{ and } \varphi_k(k)\in R_i,\\
			p_1, & \text{if } \varphi_k(k)\!\uparrow.
		\end{cases}
	\]
	Since every $z\in\omega$ belongs to precisely one set $R_i$, the function $f$ is well-de\-fi\-ned. Moreover, the graph of $f$ is a $\Delta^1_n$ set (recall that each $R_i$ is $\Delta^1_n$).
	
	We define the desired numbering $\mu$ as follows:
	\[
		\xi(k) := \nu(f(k)),\quad
		\mu := \nu \oplus \xi.
	\]
	First, note that $x\in \xi(k)$ iff $\exists t[ f(k) = t \,\&\, x\in \nu(t)]$. This shows that both numberings $\xi$ and $\mu$ are $E^1_n$-computable. Moreover, it is evident that $\mu$ indexes precisely the family $\mathcal{S}$.
	
	Towards a contradiction, assume that $\mu \leq \nu$. Then one can choose a total computable function $\varphi_e(x)$ with $\xi = \nu \circ \varphi_e$. Find the number $i\leq l$ such that $\varphi_e(e)$ belongs to $R_i$. Since $\varphi_e(e) \in R_i\subseteq Q_i$, the set $\xi(e) = \nu(\varphi_e(e))$ must contain $F_i$. On the other hand, we have $f(e) = p_{\sigma(i)}$ and $\xi(e) = \nu(f(e)) = A_{\sigma(i)}$. The set $A_{\sigma(i)}$ is $\subseteq$-minimal inside $\mathcal{S}$, and hence, $A_{\sigma(i)}\not\supseteq F_i$. This gives a contradiction, therefore, $\mu$ is not reducible to $\nu$. 	
	Theorem~\ref{theo:greatest} is proved.
\end{proof}

A careful analysis of the proof above (see also Section~\ref{sect:PD}) reveals the following:
\begin{corollary}
	For the classes $\Pi^1_1$ and $\Sigma^1_2$, Theorem~\ref{theo:greatest} holds even without assuming \textbf{PD}.
\end{corollary}

Furthermore, by applying the operator $\mathrm{Dual}$ from Lemma~\ref{lem:dual}, one can prove:

\begin{corollary}[\textbf{PD}] \label{corol:PD-fin}
	Let $n$ be a non-zero natural number, and let $\mathcal{S}$ be a finite family of $\Upsilon^1_n$ sets.
	Then the semilattice $\mathcal{R}_{\Upsilon^1_n}(\mathcal{S})$ has the greatest element if and only if the family $\mathcal{S}$ contains a greatest element under $\subseteq$.
\end{corollary}

\subsection{Minimal covers}

The technique developed in the proof of Theorem~\ref{theo:greatest} allows us to obtain a further result, which deals with minimal covers in Rogers semilattices.

Let $\mathcal{A} = (A;\leq, \vee)$ be an upper semilattice, and $a$ be an element from $\mathcal{A}$. An element $b\in\mathcal{A}$ is called a \emph{minimal cover} of $a$ if $a<b$ and there is no $c$ with $a<c<b$.

\begin{theorem}[\textbf{PD}] \label{theo:min-cover}
	Let $n$ be a non-zero natural number.
	Let \[
		\mathcal{S} = \{ A_1, A_2,\dots, A_m\}\]
	be a finite family of $E^1_n$ sets such that $\mathcal{S}$ has no least element under $\subseteq$.
	Then every element from the Rogers semilattice $\mathcal{R}^1_n(\mathcal{S})$ has a minimal cover.
	In particular, $\mathcal{R}^1_n(\mathcal{S})$ is upwards dense.
\end{theorem}

Before giving the proof of the theorem, we obtain the following auxiliary fact:

\begin{proposition}\label{prop:auxil}
	Let $\mathcal{T}$ be an arbitrary $E^1_n$-computable family, and let $\nu$ be a $E^1_n$-computable numbering of $\mathcal{T}$. Suppose that there exists a total function $f\colon \omega \to \omega$ such that:
	\begin{itemize}
		\item the graph of $f$ is $\Delta^1_n$, and
		
		\item $\nu(k) \neq \nu(f(k))$ for all $k\in\omega$.
	\end{itemize}
	Then the (degree of the) numbering $\nu$ has a minimal cover inside $\mathcal{R}^1_n(\mathcal{T})$.
\end{proposition}
\begin{proof}
	The proof mimics the idea from Theorem~2 of Ref.~\cite{BP-02}. Let $M$ be a maximal computably enumerable set. Fix a total computable, injective function $g$ such that $range(g) = M$. Assume that $\overline{M}:=\omega\setminus M = \{ m_0 <_{\omega} m_1 <_{\omega} m_2 <_{\omega} \dots\}$.
	
	Define a numbering $\mu$ as follows:
	\[
		\mu(k) := \begin{cases}
			\nu (g^{-1}(k)), & \text{if } k\in M,\\
			\nu(0), & \text{if } k = m_e \text{ and } k\not\in dom(\varphi_e),\\
			\nu(f(\varphi_e(k))), & \text{if } k = m_e \text{ and } k \in dom(\varphi_e).
		\end{cases}
	\]
	Via standard counting of quantifiers, one can see that the numbering $\mu$ is $E^1_n$-computable: e.g., the condition $x\in \nu(f(\varphi_e(k)))$ is equivalent to a $E^1_n$ formula
	\[
		\exists u \exists v[ \varphi_e(k)\!\downarrow\ =u \,\&\, f(u) = v \,\&\,  x\in \nu(v)].
	\]
	
	Clearly, $\nu$ is equal to $\mu\circ g$; hence, $\nu \leq \mu$ and $\mu$ indexes precisely the family $\mathcal{T}$. Assume that a total computable function $\varphi_e$ reduces $\mu$ to $\nu$. Then $\mu(m_e) = \nu(\varphi_e(m_e))$ and on the other hand, $\mu(m_e) = \nu(f(\varphi_e(m_e))) \neq \nu(\varphi_e(m_e))$, which gives a contradiction. Thus, we deduce that $\mu$ is not reducible to $\nu$.
	
	Now assume that $\xi$ is a numbering of $\mathcal{T}$ such that $\nu \leq \xi \leq \mu$. Fix total computable functions $h$ and $p$ such that $\xi = \mu\circ h$ and $\nu = \xi\circ p$. Since the set $M$ is maximal, one of the following two cases holds:
	
	\emph{Case 1.} The set $range(h)\setminus M$ is finite. Assume that $range(h)\setminus M = \{ b_0,b_1,\dots,b_t\}$ and choose $\nu$-indices $a_i$, $i\leq t$, such that $\nu(a_i) = \mu(b_i)$. Define a computable function
	\[
		q(k) := \begin{cases}
			a_i, & \text{if } h(k) = b_i \text{ for some } i\leq t,\\
			g^{-1}(h(k)), & \text{otherwise}.
		\end{cases}
	\]
	The function $q$ is well-defined, and it reduces $\xi$ to $\nu$; hence, $\xi \equiv \nu$.
	
	\emph{Case 2.} The set $\overline{M}\setminus range(h)$ is finite. Assume that $\overline{M}\setminus range(h) = \{ c_0,c_1,\dots,c_t\}$, and choose $\xi$-indices $d_i$, $i \leq t$ such that $\xi(d_i) = \mu(c_i)$. Define a computable function $q$ according to the following rules:
	\begin{itemize}
		\item[(a)] If $k = c_i$, then set $q(k) := d_i$.
		
		\item[(b)] Otherwise, find the least step $s$ such that one of the following holds:
		\begin{itemize}
			\item[(b.1)] There is a number $l$ such that $h(l)[s]\!\downarrow\ = k$. Then set $q(k):=l$.
			
			\item[(b.2)] The number $k$ belongs to $M[s]$. Then set $q(k):= p(g^{-1}(k))$.
		\end{itemize}
	\end{itemize}
	It is not hard to show that the function $q$ is well-defined, and $\mu = \xi\circ q$, hence $\xi\equiv \mu$.
	
	Summarizing, we showed that there are no numberings strictly between $\nu$ and $\mu$, and hence, $\mu$ is a mininal cover for $\nu$. Proposition~\ref{prop:auxil} is proved.
\end{proof}

\begin{proof}[Proof of Theorem~\ref{theo:min-cover}]
We follow the notations employed by the proof of the direction $(\Rightarrow)$ in Theorem~\ref{theo:greatest}: The sets $A_1,A_2,\dots,A_l$ are chosen as all $\subseteq$-mi\-ni\-mal elements from $\mathcal{S}$. Given a $E^1_n$-computable numbering $\nu$ of $\mathcal{S}$, we introduce precisely the same objects $p_i$, $p_{\sigma(i)}$, $Q_i$, and $R_i$ as in Theorem~\ref{theo:greatest}.

We define a total function $f\colon \omega \to \omega$ as follows:
\[
	f(k) := p_{\sigma(i)}, \text{ if } k \in R_i.
\]
Clearly, $f$ is well-defined, and the graph of $f$ is $\Delta^1_n$.

Suppose that $k\in R_i$. Then $A_i \subseteq \nu(k)$. On the other hand, $\nu(f(k)) = \nu(p_{\sigma(i)}) = A_{\sigma(i)}$ and hence, $\nu(f(k))\not\supseteq A_i$. Thus, $\nu(f(k)) \neq \nu(k)$. By Proposition~\ref{prop:auxil}, we deduce that the numbering $\nu$ has a minimal cover inside $\mathcal{R}^1_n(\mathcal{S})$. Theorem~\ref{theo:min-cover} is proved.
\end{proof}


\section{Rogers semilattices of infinite families} \label{sect:inf_fam}

In this section, we discuss elementary properties of the semilattices $\mathcal{R}^1_n(\mathcal{S})$, where $\mathcal{S}$ is infinite. The first property, which differs from the results of Section~\ref{sect:finite_fam}, is the following fact: $\mathcal{R}^1_n(\mathcal{S})$ never contains the least element, as witnessed by the result of Dorzhieva:

\begin{proposition}[Dorzhieva, Corollary~1 in Ref.~\cite{Dor-16}] \label{prop:Dor}
	Let $n$ be a non-zero natural number, and let $\mathcal{S}$ be an infinite $\Pi^1_n$-computable family. Then the Rogers semilattice $\mathcal{R}_{\Pi^1_n}(\mathcal{S})$ contains infinitely many minimal elements.
\end{proposition}

Note that Lemma~\ref{lem:dual} implies the following: if one replaces $\Pi^1_n$ by $\Sigma^1_n$, then Proposition~\ref{prop:Dor} still stays true. Moreover, by combining Propositions~\ref{prop:inf-not-lat} and~\ref{prop:Dor}, we obtain:

\begin{corollary}\label{cor:Dor}
	Let $\mathcal{S}$ be an arbitrary $E^1_n$-computable family, which contains at least two elements. Then the Rogers semilattice $\mathcal{R}^1_n(\mathcal{S})$ is infinite, and it is not a lattice.
\end{corollary}

Furthermore, Dorzhieva extended a result of Podzorov (Theorem~1 of Ref.~\cite{Podz-03}) and proved the following:

\begin{proposition}[Dorzhieva, a part of Lemma~1 in Ref.~\cite{Dor-16}] \label{prop:ideal-with-no-min}
	Let $\mathcal{S}$ be an infinite $\Pi^1_n$-computable family. Then there is a principal ideal $\mathcal{I}$ inside $\mathcal{R}_{\Pi^1_n}(\mathcal{S})$, which is isomorphic to the upper semilattice $\mathcal{E}^{\ast} \setminus \{ \bot\}$ (i.e. the semilattice of all c.e. sets under set-theoretic almost-inclusion $\subseteq^{\ast}$, with the least element omitted).
\end{proposition}

We make use of Proposition~\ref{prop:ideal-with-no-min} and exhibit a further elementary property, which differs from Section~\ref{sect:finite_fam}.

\begin{definition}[Definition~3.2 and Proposition~3.2 of Ref.~\cite{BGS-03-isom-3}]
	Let $\mathcal{A} = (A;\leq,\vee)$ be an upper semilattice. We say that $\mathcal{A}$ is \emph{weakly distributive} if it satisfies the following: if one adds an external least element $\bot$ to $\mathcal{A}$, then the resulting structure is a distributive upper semilattice. Equivalently, $\mathcal{A}$ is weakly distributive if and only if for all $b,a_0,a_1\in\mathcal{A}$, the following holds:
	\begin{multline*}
		(b\leq a_0 \vee a_1)\,\&\,(b\nleq a_0)\,\&\,(b\nleq a_1) \ \Rightarrow\\
		\exists b_0 \exists b_1 [ (b = b_0 \vee b_1) \,\&\, (b_0 \leq a_0) \,\&\,  (b_1 \leq a_1)].
	\end{multline*}
\end{definition}

\begin{theorem}\label{theo:not-weak-distr}
	Let $n$ be a non-zero natural number, and let $\mathcal{S}$ be an infinite $E^1_n$-computable family. Then the semilattice $\mathcal{R}^1_n(\mathcal{S})$ is not weakly distributive.
\end{theorem}
\begin{proof}
	The proof follows the ideas of Theorem~3.2 in Ref.~\cite{BGS-03-isom-3}. We will build three $E^1_n$-computable numberings $\mu$, $\nu_0$, and $\nu_1$ of the family $\mathcal{S}$, which witness the failure of weak distributivity.
	
	First, we define auxiliary numberings $\alpha$, $\beta$, and $\gamma$. Let $\alpha$ be a $E^1_n$-com\-pu\-ta\-ble numbering of $\mathcal{S}$ such that the principal ideal $\mathcal{I}$, induced by $\alpha$ inside $\mathcal{R}^1_n(\mathcal{S})$, contains no minimal elements. The existence of such $\alpha$ follows from Proposition~\ref{prop:ideal-with-no-min} and Lemma~\ref{lem:dual}.
	
	Fix a maximal c.e. set $M$ and an element $A$ from $\mathcal{S}$. Assume that $\overline{M}:= \omega\setminus M = \{m_0<_{\omega} m_1 <_{\omega} m_2 <_{\omega} \dots\}$. Define
	\[
		\beta(k) := \begin{cases}
			\alpha(e), & \text{if } k=m_e \text{ for some } e\in\omega,\\
			A, & \text{if } k\in M.
		\end{cases}
	\]
	Clearly, $\beta$ is a $E^1_n$-computable numbering of $\mathcal{S}$. We claim that $\beta$ is a \emph{minimal} numbering of $\mathcal{S}$, i.e. $\beta$ induces a minimal element inside $\mathcal{R}^1_n(\mathcal{S})$. Indeed, suppose that $\xi$ is a numbering of $\mathcal{S}$, and a computable function $f$ reduces $\xi$ to $\beta$. Then the maximality of $M$ implies that the set $range(f)\setminus \overline{M}$ must be finite. This allows us to build a function $g$, which will reduce $\beta$ to $\xi$, as follows:
	\begin{itemize}
		\item[(a)] If $k\in range(f)\setminus \overline{M}$, then one can choose an appropriate value $g(k)$ in a non-uniform way.
		
		\item[(b)] If $k\in M\cup range(f)$, then the image $g(k)$ will be defined either as some $l\in f^{-1}(k)$, or as a fixed $a$ with $\xi(a)=A$.
	\end{itemize}
	A formal construction of the desired $g$ can be recovered from Case~2 in Proposition~\ref{prop:auxil}, or from Theorem~1.3 in Ref.~\cite{BGS-03-completeness}. Recall that the principal ideal $\mathcal{I}$, induced by $\alpha$, has no minimal elements. Hence, $\beta\nleq \alpha$.
	
	Fix different elements $B$ and $C$ from $\mathcal{S}$ such that $B\neq A \neq C$. Define
	\[
		\gamma(k) := \begin{cases}
			B, & \text{if } k\in \emptyset',\\
			C, & \text{if } k\not\in \emptyset'.
		\end{cases}
	\]
	Note that $\gamma$ indexes \emph{only} the finite family $\{ B, C\}$. Thus, clearly, $\alpha\nleq \gamma$ and $\beta\nleq \gamma$.
	
	\begin{claim}
		$\gamma$ is not reducible to $\beta$.
	\end{claim}
	\begin{proof}
		Assume that a function $f$ reduces $\gamma$ to $\beta$. Then $range(f) \subseteq \overline{M}$. Since the family $\mathcal{S}$ is infinite and $M$ is maximal, we deduce that the set $range(f)$ must be finite, but this contradicts the non-computablity of the set $\emptyset'$.
	\end{proof}
	
	From now on, we will employ the following useful fact without explicitly referencing it:	
	\begin{lemma}[essentially Proposition~3.1 in Ref.~\cite{BGS-03-isom-3}] \label{lem:decomposition}
		Let $\zeta,\xi_0,\xi_1$ be arbitrary numberings. If $\zeta \leq \xi_0\oplus \xi_1$, then at least one of the following conditions holds:
		\begin{enumerate}
			\item $\zeta \leq \xi_0$.
			
			\item $\zeta \leq \xi_1$.
			
			\item There are numberings $\zeta_0$ and $\zeta_1$ such that $\zeta_0\leq \xi_0$, $\zeta_1\leq \xi_1$, and $\zeta \equiv \zeta_0 \oplus \zeta_1$. Moreover, if the numberings $\zeta$, $\xi_0$, and $\xi_1$ are $E^1_n$-com\-pu\-ta\-ble, then both $\zeta_0$ and $\zeta_1$ are also $E^1_n$-computable.
		\end{enumerate}
	\end{lemma}
	
	Note that a similar fact has been already used in the proof of Proposition~\ref{prop:distr}.
	
	We define the desired $E^1_n$-computable numberings of $\mathcal{S}$:
	\[
		\nu_0 := \alpha \oplus \gamma, \quad \nu_1 := \beta, \quad \mu:=\gamma\oplus \beta.
	\]
	Clearly, $\mu \leq \nu_0 \oplus \nu_1$.
	
	\begin{claim}
		$\mu\nleq \nu_0$ and $\mu\nleq \nu_1$.
	\end{claim}
	\begin{proof}
		Since $\gamma\nleq \beta$, we deduce that $\mu\nleq \nu_1$. Towards a contradiction, assume that $\gamma\oplus \beta \leq \alpha\oplus\gamma$. Then $\beta \leq \alpha\oplus \gamma$. Since $\beta\nleq\alpha$ and $\beta\nleq \gamma$, there are numberings $\beta_0$ and $\beta_1$ with $\beta \equiv \beta_0\oplus \beta_1$, $\beta_0\leq \alpha$, and $\beta_1\leq \gamma$.
		
		Clearly, any set $X\in \mathcal{S}\setminus \{ B,C\}$ has a $\beta_0$-index. Moreover, by putting
		\[
			\widetilde{\beta}_0(k) := \begin{cases}
				B, & \text{if } k = 0,\\
				C, & \text{if } k = 1,\\
				\beta_0(k-2), & \text{if } k\geq 2,
			\end{cases}
		\]
		we obtain that the numbering $\widetilde{\beta}_0$ indexes \emph{the whole} family $\mathcal{S}$, $\beta \equiv \widetilde{\beta}_0 \oplus \beta_1$, and $\widetilde{\beta}_0 \leq \alpha$. The minimality of $\beta$ implies that $\widetilde{\beta}_0 \equiv \beta$. Hence, $\beta \leq \alpha$, which gives a contradiction.		
	\end{proof}
	
	Assume, towards a contradiction, that (the degrees of) $\mu$, $\nu_0$, and $\nu_1$ satisfy the weak distributivity property. Then there are $E^1_n$-com\-pu\-ta\-ble numberings $\mu_0$ and $\mu_1$ of $\mathcal{S}$ such that $\mu \equiv \mu_0 \oplus \mu_1$, $\mu_0 \leq \nu_0$, and $\mu_1 \leq \nu_1$. Since $\nu_1=\beta$ is minimal, we have $\mu_1\equiv \beta$.
	
	Clearly, $\mu_0\nleq \gamma$. Define a new numbering $\alpha_0$ of the family $\mathcal{S}$ as follows:
	\begin{itemize}
		\item[(a)] If $\mu_0 \leq \alpha$, then set $\alpha_0:=\mu_0$.
		
		\item[(b)] Otherwise, there are numberings $\alpha^{\ast}$ and $\gamma^{\ast}$ such that $\mu_0\equiv \alpha^{\ast} \oplus \gamma^{\ast}$, $\alpha^{\ast} \leq \alpha$, and $\gamma^{\ast} \leq \gamma$. Put
		\[
			\alpha_0(k) := \begin{cases}
				B, & \text{if } k = 0,\\
				C, & \text{if } k = 1,\\
				\alpha^{\ast}(k-2), & \text{if } k\geq 2.
			\end{cases}
		\]
	\end{itemize}
	Clearly, in each of the cases~(a) and~(b), $\alpha_0$ is reducible to both $\alpha$ and $\mu_0$.
	
	Recall that $\mu_0 \leq \mu = \gamma\oplus \beta$ and hence, $\alpha_0 \leq \gamma\oplus \beta$. Obviously, $\alpha_0\nleq \gamma$. We define a numbering $\alpha_1$ of $\mathcal{S}$:
	\begin{itemize}
		\item[(a)] If $\alpha_0 \leq \beta$, then set $\alpha_1:=\alpha_0$.
		
		\item[(b)] Otherwise, there are numberings $\gamma'$ and $\beta'$ such that $\alpha_0\equiv \gamma' \oplus \beta'$, $\gamma' \leq \gamma$, and $\beta' \leq \beta$. Set
		\[
			\alpha_1(k) := \begin{cases}
				B, & \text{if } k = 0,\\
				C, & \text{if } k = 1,\\
				\beta'(k-2), & \text{if } k\geq 2.
			\end{cases}
		\]
	\end{itemize}
	Clearly, we have $\alpha_1\leq \alpha_0$ and $\alpha_1 \leq \beta$.
	
	Since $\beta$ is minimal, we deduce that $\beta \equiv \alpha_1 \leq \alpha_0 \leq \alpha$, which contradicts the original choice of the numbering $\alpha$. Therefore, the numberings $\mu,\nu_0,\nu_1$ witness the failure of weak distributivity. Theorem~\ref{theo:not-weak-distr} is proved.
\end{proof}



\section{Further discussion} \label{sect:further}

After all the results of previous sections, it is completely possible that an interested reader would ask the following natural question:

\begin{problem}
	Let $\Gamma$ be a class of the analytical hierarchy. What results on Rogers semilattices of $\Gamma$-computable families can be obtained, if one replaces \textbf{PD} with another set-theoretic assumption?
\end{problem}

Here we give a (very brief) case study for this problem: We assume the \emph{Axiom of Constructibility} ($V=L$) and list some of results, which can be obtained under this assumption.

The Axiom of Constructibility says that every set is constructible. A formal statement of the axiom can be found, e.g., in Chap.~13 of Ref.~\cite{Jech}.

Recall that the key property of a class $E^1_n$, which was heavily employed in the previous sections, is the prewellordering property (see Section~\ref{subsub:prewell}).

\begin{theorem}[see Exercises 5A.3 and 4B.10 of Ref.~\cite{Moschovakis}] \label{theo:Godel}
	Assume $(V=L)$. For every $n\geq 3$, the class $\Sigma^1_n$ has the prewellordering property. Consequently, $\Sigma^1_n$ satisfies the reduction principle.
\end{theorem}

Therefore, one can repeat the proofs of Theorems~\ref{theo:greatest} and~\ref{theo:min-cover} verbatim, and obtain the following:

\begin{corollary}[$\mathbf{V=L}$]
	Let $n\geq 3$, and let $\mathcal{S}$ be a finite family of $\Sigma^1_n$ sets.
	\begin{enumerate}
		\item The Rogers semilattice $\mathcal{R}_{\Sigma^1_n}(\mathcal{S})$ has the greatest element if and only if the family $\mathcal{S}$ contains a least element under $\subseteq$.
		
		\item  If $\mathcal{R}_{\Sigma^1_n}(\mathcal{S})$ has no greatest element, then every element from $\mathcal{R}_{\Sigma^1_n}(\mathcal{S})$ has a minimal cover.
	\end{enumerate}
\end{corollary}

In particular, we observe the following simple fact, which is still interesting on its own:
\begin{remark}
	Let $\mathcal{S}$ be a finite family of $\Sigma^1_3$ sets.
	\begin{itemize}
		\item[(a)] If one assumes \textbf{PD}, then $\mathcal{R}_{\Sigma^1_3}(\mathcal{S})$ has a greatest element iff $\mathcal{S}$ contains a \emph{greatest} element under $\subseteq$ (Corollary~\ref{corol:PD-fin}).
		
		\item[(b)] If one assumes $(V=L)$, then $\mathcal{R}_{\Sigma^1_3}(\mathcal{S})$ has a greatest element iff $\mathcal{S}$ contains a \emph{least} element under $\subseteq$.
	\end{itemize}
\end{remark}

We strongly conjecture that one can employ the techniques developed by Tanaka~\cite{Tanaka} to provide a complete solution of Problem~\ref{prob:02} under \textbf{PD} and under $(V=L)$.

As a concluding remark, we note the following: It seems that all our proofs essentially employed only the properties inherent to Spector pointclasses (see Section 4C of Ref.~\cite{Moschovakis}). Hence, we formulate the following:

\begin{problem}
	Develop the theory of Rogers semilattices for Spector pointclasses.
\end{problem}


\bibliographystyle{plain}
\bibliography{Numberings}

\begin{thebibliography}{10}

\bibitem{Add-59}
J.~W. Addison.
\newblock Separation principles in the hierarchies of classical and effective
  descriptive set theory.
\newblock {\em Fundam. Math.}, 46(2):123--135, 1959.

\bibitem{Add-Mos-68}
J.~W. Addison and Y.~N. Moschovakis.
\newblock Some consequences of the axiom of definable determinateness.
\newblock {\em Proc. Natl. Acad. Sci. USA}, 59(3):708--712, 1968.

\bibitem{BG-08}
S.~Badaev and S.~Goncharov.
\newblock Computability and numberings.
\newblock In S.~B. Cooper, B.~L{\"o}we, and A.~Sorbi, editors, {\em New
  Computational Paradigms}, pages 19--34. Springer, New York, 2008.

\bibitem{BGPS-03}
S.~Badaev, S.~Goncharov, S.~Podzorov, and A.~Sorbi.
\newblock Algebraic properties of {R}ogers semilattices of arithmetical
  numberings.
\newblock In S.~S. Goncharov and S.~B. Cooper, editors, {\em Computability and
  Models}, pages 45--78. Springer, New York, 2003.

\bibitem{BGS-03-completeness}
S.~Badaev, S.~Goncharov, and A.~Sorbi.
\newblock Completeness and universality of arithmetical numberings.
\newblock In S.~S. Goncharov and S.~B. Cooper, editors, {\em Computability and
  Models}, pages 11--44. Springer, New York, 2003.

\bibitem{BGS-03-isom-3}
S.~Badaev, S.~Goncharov, and A.~Sorbi.
\newblock Isomorphism types and theories of {R}ogers semilattices of
  arithmetical numberings.
\newblock In S.~S. Goncharov and S.~B. Cooper, editors, {\em Computability and
  Models}, pages 79--92. Springer, New York, 2003.

\bibitem{BG-00}
S.~A. Badaev and S.~S. Goncharov.
\newblock Theory of numberings: open problems.
\newblock In P.~Cholak, S.~Lempp, M.~Lerman, and R.~Shore, editors, {\em
  Computability Theory and Its Applications}, volume 257 of {\em Contemp.
  Math.}, pages 23--38. American Mathematical Society, Providence, 2000.

\bibitem{BGS-05}
S.~A. Badaev, S.~S. Goncharov, and A.~Sorbi.
\newblock Elementary theories for {R}ogers semilattices.
\newblock {\em Algebra Logic}, 44(3):143--147, 2005.

\bibitem{BGS-06}
S.~A. Badaev, S.~S. Goncharov, and A.~Sorbi.
\newblock Isomorphism types of {R}ogers semilattices for families from
  different levels of the arithmetical hierarchy.
\newblock {\em Algebra Logic}, 45(6):361--370, 2006.

\bibitem{BP-02}
S.~A. Badaev and S.~{Yu.} Podzorov.
\newblock Minimal coverings in the {R}ogers semilattices of
  {$\Sigma^0_n$}-computable numberings.
\newblock {\em Sib. Math. J.}, 43(4):616--622, 2002.

\bibitem{BMY}
N.~Bazhenov, M.~Mustafa, and M.~Yamaleev.
\newblock Elementary theories and hereditary undecidability for semilattices of
  numberings.
\newblock {\em Arch. Math. Logic}, 58(3--4):485--500, 2019.

\bibitem{BOY}
N.~Bazhenov, S.~Ospichev, and M.~Yamaleev.
\newblock Isomorphism types of {R}ogers semilattices in the analytical
  hierarchy.
\newblock Preprint. arXiv:1912.05226.

\bibitem{Dor-14}
M.~V. Dorzhieva.
\newblock Elimination of metarecursive in {O}wing's theorem.
\newblock {\em Vestn. Novosib. Gos. Univ., Ser. Mat. Mekh. Inform.},
  14(1):35--43, 2014.
\newblock In Russian.

\bibitem{Dor-16}
M.~V. Dorzhieva.
\newblock Undecidability of elementary theory of {R}ogers semilattices in
  analytical hierarchy.
\newblock {\em Sib. Elektron Mat. Izv.}, 13:148--153, 2016.
\newblock In Russian.

\bibitem{Dor-18}
M.~V. Dorzhieva.
\newblock A single-valued numbering for the family of all {$\Sigma^1_2$}-sets.
\newblock {\em Sib. Zh. Chist. Prikl. Mat.}, 18(2):47--52, 2018.
\newblock In Russian.

\bibitem{Ershov-Book}
{Yu.}~L. Ershov.
\newblock {\em Theory of numberings}.
\newblock Nauka, Moscow, 1977.
\newblock In Russian.

\bibitem{Ershov-99}
{Yu.}~L. Ershov.
\newblock Theory of numberings.
\newblock In E.~R. Griffor, editor, {\em Handbook of Computability Theory},
  volume 140 of {\em Stud. Logic Found. Math.}, pages 473--503. North-Holland,
  Amsterdam, 1999.

\bibitem{Ershov-03}
{Yu.}~L. Ershov.
\newblock Necessary isomorphism conditions for {R}ogers semilattices of finite
  partially ordered sets.
\newblock {\em Algebra Logic}, 42(4):232--236, 2003.

\bibitem{EL-73}
{Yu.}~L. Ershov and I.~A. Lavrov.
\newblock The upper semilattice {$L(\gamma)$}.
\newblock {\em Algebra Logic}, 12(2):93--106, 1973.

\bibitem{Ershov-1}
{Ju.}~L. Er{\v{s}}ov.
\newblock Theorie der {N}umerierungen {I}.
\newblock {\em Z. Math. Logik Grundlagen Math.}, 19(19--25):289--388, 1973.

\bibitem{Ershov-2}
{Ju.}~L. Er{\v{s}}ov.
\newblock Theorie der {N}umerierungen {II}.
\newblock {\em Z. Math. Logik Grundlagen Math.}, 21(1):473--584, 1975.

\bibitem{Ershov-3}
{Ju.}~L. Er{\v{s}}ov.
\newblock Theorie der {N}umerierungen {III}.
\newblock {\em Z. Math. Logik Grundlagen Math.}, 23(19--24):289--371, 1977.

\bibitem{Goedel}
K.~G{\"o}del.
\newblock {\"U}ber formal unentscheidbare {S}{\"a}tze der {P}rincipia
  {M}athematica und verwandter {S}ysteme. {I.}
\newblock {\em Monatsh. Math. Phys.}, 38(1):173--198, 1931.

\bibitem{GS-97}
S.~S. Goncharov and A.~Sorbi.
\newblock Generalized computable numerations and nontrivial {R}ogers
  semilattices.
\newblock {\em Algebra Logic}, 36(6):359--369, 1997.

\bibitem{Jech}
T.~Jech.
\newblock {\em Set Theory. The Third Millenium Edition, revised and expanded}.
\newblock Springer, Berlin, 2002.

\bibitem{KPF-18}
I.~{Sh.} Kalimullin, V.~G. Puzarenko, and M.~{Kh.} Faizrakhmanov.
\newblock Positive presentations of families in relation to reducibility with
  respect to enumerability.
\newblock {\em Algebra Logic}, 57(4):320--323, 2018.

\bibitem{KPF-19}
I.~{Sh.} Kalimullin, V.~G. Puzarenko, and M.~{Kh.} Faizrakhmanov.
\newblock Partial decidable presentations in hyperarithmetic.
\newblock {\em Sib. Math. J.}, 60(3):464--471, 2019.

\bibitem{Khut-71}
A.~B. Khutoretskii.
\newblock On the cardinality of the upper semilattice of computable
  enumerations.
\newblock {\em Algebra Logic}, 10(5):348--352, 1971.

\bibitem{Kleene}
S.~C. Kleene.
\newblock {\em Introduction to metamathematics}.
\newblock Van Nostrand, New York, 1952.

\bibitem{Kleene-55}
S.~C. Kleene.
\newblock Hierarchies of number-theoretic predicates.
\newblock {\em Bull. Amer. Math. Soc.}, 61(3):193--213, 1955.

\bibitem{KU}
A.~N. Kolmogorov and V.~A. Uspenskii.
\newblock On the definition of an algorithm.
\newblock {\em Uspehi Mat. Nauk}, 13(4):3--28, 1958.
\newblock In Russian.

\bibitem{Martin}
D.~A. Martin.
\newblock The axiom of determinateness and reduction principles in the
  analytical hierarchy.
\newblock {\em Bull. Amer. Math. Soc.}, 74(4):687--689, 1968.

\bibitem{Mos-71}
Y.~N. Moschovakis.
\newblock Uniformization in a playful universe.
\newblock {\em Bull. Amer. Math. Soc.}, 77(5):731--736, 1971.

\bibitem{Moschovakis}
Y.~N. Moschovakis.
\newblock {\em Descriptive Set Theory. Second Edition}.
\newblock American Mathematical Society, Providence, 2009.

\bibitem{Podz-03}
S.~{Yu.} Podzorov.
\newblock Initial segments in {R}ogers semilattices of
  {$\Sigma^0_n$}-computable numberings.
\newblock {\em Algebra Logic}, 42(2):121--129, 2003.

\bibitem{Podz}
S.~{Yu.} Podzorov.
\newblock Arithmetical ${D}$-degrees.
\newblock {\em Sib. Math. J.}, 49(6):1109--1123, 2008.

\bibitem{Rogers}
H.~Rogers.
\newblock G{\"o}del numberings of partial recursive functions.
\newblock {\em J. Symb. Logic}, 23(3):331--341, 1958.

\bibitem{Rogers-book}
H.~Rogers.
\newblock {\em Theory of recursive functions and effective computability}.
\newblock McGraw-Hill, New York, 1967.

\bibitem{Sacks}
G.~E. Sacks.
\newblock {\em Higher recursion theory}.
\newblock Springer, Berlin, 1990.

\bibitem{Sel-76}
V.~L. Selivanov.
\newblock Two theorems on computable numberings.
\newblock {\em Algebra Logic}, 15(4):297--306, 1976.

\bibitem{Tanaka}
H.~Tanaka.
\newblock Recursion theory in analytical hierarchy.
\newblock {\em Comment. Math. Univ. St. Pauli}, 27(2):113--132, 1978.

\bibitem{Usp-55}
V.~A. Uspenskii.
\newblock Systems of denumerable sets and their enumeration.
\newblock {\em Dokl. Akad. Nauk SSSR}, 105:1155--1158, 1958.
\newblock In Russian.

\bibitem{V'yugin}
V.~V. V'yugin.
\newblock On some examples of upper semilattices of computable enumerations.
\newblock {\em Algebra Logic}, 12(5):287--296, 1973.

\end{thebibliography}

\end{document}